\documentclass[reqno]{amsart}
\usepackage{amsmath, amssymb, amsthm, epsfig}
\usepackage{hyperref, latexsym}
\usepackage[mathscr]{euscript}

\usepackage{color}
\usepackage{fullpage} 
\usepackage{setspace}

\begin{document}
 \bibliographystyle{plain}

 \newtheorem{theorem}{Theorem}
 \newtheorem{lemma}[theorem]{Lemma}
 \newtheorem{proposition}[theorem]{Proposition}
 \newtheorem{corollary}[theorem]{Corollary}
 \theoremstyle{definition}
 \newtheorem{definition}[theorem]{Definition}
 \newtheorem{example}[theorem]{Example}
 \theoremstyle{remark}
 \newtheorem{remark}[theorem]{Remark}
 \newcommand{\mc}{\mathcal}
 \newcommand{\A}{\mc{A}}
 \newcommand{\B}{\mc{B}}
 \newcommand{\cc}{\mc{C}}
 \newcommand{\D}{\mc{D}}
 \newcommand{\E}{\mathbb{E}}
 \newcommand{\F}{\mc{F}}
 \newcommand{\G}{\mc{G}}
 \newcommand{\sH}{\mc{H}}
 \newcommand{\I}{\mc{I}}
 \newcommand{\J}{\mc{J}}
 \newcommand{\nn}{\mc{N}}
 \newcommand{\rr}{\mc{R}}
 \newcommand{\sS}{\mc{S}}
 \newcommand{\U}{\mc{U}}
 \newcommand{\X}{\mc{X}}
 \newcommand{\Y}{\mc{Y}}
 \newcommand{\C}{\mathbb{C}}
 \newcommand{\R}{\mathbb{R}}
 \newcommand{\N}{\mathbb{N}}
 \newcommand{\Q}{\mathbb{Q}}
 \newcommand{\Z}{\mathbb{Z}}
 \newcommand{\csch}{\mathrm{csch}}
 \newcommand{\tF}{\widehat{F}}
 \newcommand{\tG}{\widehat{G}}
 \newcommand{\tH}{\widehat{H}}
 \newcommand{\tf}{\widehat{f}}
 \newcommand{\ug}{\widehat{g}}
 \newcommand{\wg}{\widetilde{g}}
 \newcommand{\uh}{\widehat{h}}
 \newcommand{\wh}{\widetilde{h}}
 \newcommand{\wl}{\widetilde{l}}
 \newcommand{\tk}{\widehat{k}}
 \newcommand{\tK}{\widehat{K}}
 \newcommand{\tl}{\widehat{l}}
 \newcommand{\tL}{\widehat{L}}
 \newcommand{\tm}{\widehat{m}}
 \newcommand{\tM}{\widehat{M}}
 \newcommand{\tp}{\widehat{\varphi}}
 \newcommand{\tq}{\widehat{q}}
 \newcommand{\tT}{\widehat{T}}
 \newcommand{\tU}{\widehat{U}}
 \newcommand{\tu}{\widehat{u}}
 \newcommand{\tV}{\widehat{V}}
 \newcommand{\tv}{\widehat{v}}
 \newcommand{\tW}{\widehat{W}}
 \newcommand{\ba}{\boldsymbol{a}}
 \newcommand{\bal}{\boldsymbol{\alpha}}
 \newcommand{\bx}{\boldsymbol{x}}
 \newcommand{\p}{\varphi}
 \newcommand{\f}{\frac52}
 \newcommand{\g}{\frac32}
 \newcommand{\h}{\frac12}
 \newcommand{\hh}{\tfrac12}
 \newcommand{\ds}{\text{\rm d}s}
 \newcommand{\dt}{\text{\rm d}t}
  \newcommand{\dr}{\text{\rm d}r}
 \newcommand{\du}{\text{\rm d}u}
 \newcommand{\dv}{\text{\rm d}v}
 \newcommand{\dw}{\text{\rm d}w}
  \newcommand{\dz}{\text{\rm d}z}
 \newcommand{\dx}{\text{\rm d}x}
   \newcommand{\dxx}{\text{\rm d}{\bf x}}
 \newcommand{\dy}{\text{\rm d}y}
 \newcommand{\dl}{\text{\rm d}\lambda}
 \newcommand{\dmu}{\text{\rm d}\mu(\lambda)}
 \newcommand{\dnu}{\text{\rm d}\nu(\lambda)}
  \newcommand{\dnuN}{\text{\rm d}\nu_N(\lambda)}
\newcommand{\dnus}{\text{\rm d}\nu_{\sigma}(\lambda)}
 \newcommand{\dlnu}{\text{\rm d}\nu_l(\lambda)}
 \newcommand{\dnnu}{\text{\rm d}\nu_n(\lambda)}
\newcommand{\sech}{\text{\rm sech}}
\newcommand{\CC}{\mathbb{C}}
\newcommand{\NN}{\mathbb{N}}
\newcommand{\RR}{\mathbb{R}}
\newcommand{\ZZ}{\mathbb{Z}}
\newcommand{\thp}{\theta^+}
\newcommand{\thpn}{\theta^+_N}
\newcommand{\vthp}{\vartheta^+}
\newcommand{\vthpn}{\vartheta^+_N}
\newcommand{\ft}[1]{\widehat{#1}}
\newcommand{\support}[1]{\mathrm{supp}(#1)}
\newcommand{\gplus}{G^+_\lambda}
\newcommand{\ph}{\mathcal{H}}
\newcommand{\godd}{G^o_\lambda}
\newcommand{\mplus}{M_\lambda^+}
\newcommand{\lplus}{L_\lambda^+}
\newcommand{\modd}{M_\lambda^o}
\newcommand{\lodd}{L_\lambda^o}
\newcommand{\sgp}{x_+^0}
\newcommand{\Tl}{T_\lambda}
\newcommand{\Tlc}{T_{\lambda,c}}
\newcommand{\El}{{E_\lambda}}

 \newcommand{\im}{{\rm Im}\,}
  \newcommand{\re}{{\rm Re}\,}

 \newcommand{\T}{\mc{T}}
 \newcommand{\M}{\mc{M}}
 \renewcommand{\L}{\mc{L}}
 \newcommand{\K}{\mc{K}}
 \renewcommand{\H}{\mc{H}}
 \newcommand{\Bc}{\mathcal{B}}

\newcommand{\TT}{\mathfrak{T}}
\newcommand{\MM}{\mathfrak{M}}
\newcommand{\LL}{\mathfrak{L}}
\newcommand{\KK}{\mathfrak{K}}
\newcommand{\GG}{\mathfrak{G}}
\newcommand{\HH}{\mathfrak{H}}

 \renewcommand{\d}{\text{\rm d}}
 
  \renewcommand{\SS}{\mathbb{S}}
 
  \newcommand{\z}{{\bf z}}
    \newcommand{\x}{{\bf x}}
      \newcommand{\y}{{\bf y}}
        \renewcommand{\t}{{\bf t}}
         \renewcommand{\v}{{\bf v}}
    \newcommand{\xxi}{{\bf \xi}}      

 \newcommand{\ov}{\overline}   
  \newcommand{\wt}{\widetilde}

 \newcommand{\uc}{\partial \mathbb{D}}
\newcommand{\ud}{\mathbb{D}}
\newcommand{\dbpn}{\mathcal{H}_n(P)}

\newcommand{\newr}[1]{\textcolor{red}{#1}}
\newcommand{\new}[1]{\textcolor{black}{#1}}

 \def\today{\ifcase\month\or
  January\or February\or March\or April\or May\or June\or
  July\or August\or September\or October\or November\or December\fi
  \space\number\day, \number\year}

\title[Extremal functions]{Extremal functions in de Branges \\
and Euclidean spaces II}
\author[Carneiro and Littmann]{Emanuel Carneiro and  Friedrich Littmann}

\date{\today}

\address{IMPA - Instituto Nacional de Matem\'{a}tica Pura e Aplicada, Estrada Dona Castorina, 110, Rio de Janeiro, Brazil 22460-320.}
\email{carneiro@impa.br}
\address{Department of mathematics, North Dakota State University, Fargo, ND 58105-5075.}
\email{friedrich.littmann@ndsu.edu}

\maketitle

\begin{abstract}
This paper presents the Gaussian subordination framework to generate optimal one-sided approximations to multidimensional real-valued functions by functions of prescribed exponential type. Such extremal problems date back to the works of Beurling and Selberg and provide a variety of applications in analysis and analytic number theory. Here we majorize and minorize (on $\R^N$) the Gaussian $\x \mapsto e^{-\pi \lambda |\x|^2}$, where $\lambda >0$ is a free parameter, by functions with distributional Fourier transforms supported on Euclidean balls, optimizing weighted $L^1$-errors. By integrating the parameter $\lambda$ against suitable measures, we solve the analogous problem for a wide class of radial functions. Applications to inequalities and periodic analogues are discussed. The constructions presented here rely on the theory of de Branges spaces of entire functions and on new interpolations tools derived from the theory of Laplace transforms of Laguerre-P\'{o}lya functions.

\end{abstract}

\numberwithin{equation}{section}

\section{Introduction}

In this paper we continue our study of the Beurling-Selberg extremal problem and its applications, in an unprecedented general setting. We develop here the {\it Gaussian subordination} method to produce optimal (in weighted $L^1$-metrics) one-sided approximations of multidimensional real-valued functions by functions of exponential type. In particular, this plainly becomes the most powerful framework to deal with this sort of extremal problems, extending a number of previous works in the literature (e.g. \cite{CLV, CV2, GV} in the one-dimensional case, and \cite{CL3} in the multidimensional case). 

\smallskip

We follow closely the notation of our precursor \cite{CL3} to facilitate some of the references.

\subsection{The extremal problem} Throughout the text we use the standard Fourier analysis notation as in \cite{SW}. We denote vectors in $\R^N$ or $\C^N$ with bold font (e.g. $\x$, $\y$, $\z$) and numbers in $\R$ or $\C$ with regular font (e.g. $x,y,z$). For $\z = (z_1, z_2, \ldots, z_N) \in \C^N$ we let $|\cdot|$ denote the usual norm $|\z| = (|z_1|^2 + \ldots + |z_N|^2)^{1/2}$, and define a second norm $\|\cdot\|$ by 
$$\|\z\| = \sup \left\{ \left|\sum_{n=1}^N z_n\,t_n\right|; \ \t \in \R^N \ {\rm and} \ |\t|\leq 1\right\}.$$
If $F: \C^N \to \C$ is an entire function of $N$ complex variables, which is not identically zero, we say that $F$ has {\it exponential type} if
$$\tau(F):= \limsup_{\|\z\|\to \infty} \|\z\|^{-1}\,\log|F(\z)| < \infty.$$
In this case, the number $\tau(F)$ is called the exponential type of $F$. Naturally, when $N=1$, this is the classical definition of exponential type. When $N \geq 2$, our definition is a particular case of a more general definition of exponential type with respect to a compact, convex and symmetric set $K \subset \R^N$ (cf. \cite[pp. 111-112]{SW}). In our case, this convex set $K$ is {\it the unit Euclidean ball}. 

\smallskip

We address here the following extremal problem: given a function $\mc{F}: \R^N \to \R$ and real parameters $\delta >0$ and $\nu > -1$, we seek entire functions $\mc{L}: \C^N \to \C$ and $\mc{M}: \C^N \to \C$ such that 
\smallskip
\begin{enumerate}
\item[(i)] $\mc{L}$ and $\mc{M}$ have exponential type at most $\delta$.
\smallskip
\item[(ii)] $\mc{L}$ and $\mc{M}$ are real-valued on $\R^N$ and
\begin{equation}\label{Intro_EP1}
\mc{L}(\x) \leq \mc{F}(\x) \leq \mc{M}(\x)
\end{equation}
for all $\x \in \R^N$.
\smallskip
\item[(iii)] Subject to (i) and (ii), the values of the integrals
\begin{equation}\label{Intro_EP2}
\int_{\R^N} \big\{{\mc M}(\x) - \mc{F}(\x)\big\}\, |\x|^{2\nu +2 - N}\,\d\x \ \ \ \ {\rm and} \ \ \ \  \int_{\R^N} \big\{{\mc F}(\x) - \mc{L}(\x)\big\}\, |\x|^{2\nu +2 - N}\,\d\x 
\end{equation}
are minimized.  
\end{enumerate}

In the one-dimensional case with respect to Lebesgue measure (i.e. $\nu = -1/2$), this problem and its applications have a long and rich history, dating back to the works of Beurling and Selberg for the signum function and characteristic functions of intervals, respectively (see \cite{S2,V}). The one-dimensional theory was substantially developed over the past years as we have reached the solution of this extremal problem for a wide class of even, odd and truncated functions \cite{CL, CL2, CLV, CV2, GV, L1, L3, L4, V}. The fact that these extremal majorants and minorants have compactly supported Fourier transforms (Paley-Wiener theorem) allows for a variety of interesting applications in number theory and analysis, for instance in connection to: large sieve inequalities \cite{HV, M, V}, Erd\"{o}s-Tur\'{a}n inequalities \cite{CV2, HKW, LV,V}, Hilbert-type inequalities \cite{CL3, CLV, CV2, GV, L3, V}, Tauberian theorems \cite{GV}, inequalities in signal processing \cite{DL}, and bounds in the theory of the Riemann zeta-function \cite{CC, CCM, CCM2, CS, Ga, GG}. Additional interesting works related to this theory include \cite{BMV, CV3, Gan, GL, GKM, Ke2, K2, LS, Na}.

\smallskip

The weighted one-dimensional extremal problem (including the possibility of other weights rather than the power weights) has an intimate connection to the theory of de Branges spaces of entire functions, as observed in the remarkable work of Holt and Vaaler \cite{HV} for the signum function. This was later extended in \cite{CG} for a class of odd and truncated functions, and in \cite{CL3} for a class of even functions. This extremal problem was considered in \cite{CCLM} for characteristic functions of intervals, with applications to the pair correlation of zeros of the Riemann zeta-functon (where the relevant weight is $\left\{1 - (\tfrac{\sin \pi x}{\pi x})^2\right\}\dx$). De Branges' theory has also played an important role in the recent works \cite{Ke, L4, LS}. 

\smallskip

In the multidimensional case, this problem presents additional difficulties, and has only been addressed in two occasions: in the work of Holt and Vaaler \cite{HV}, for characteristic functions of Euclidean balls, and in our precursor \cite{CL3} for the exponential function $\x \mapsto e^{-\lambda |\x|}$, where $\lambda >0$, and for a family of radial functions with exponential subordination given by
\begin{equation}\label{Exp_Sub}
\x \mapsto \int_0^{\infty}\big\{ e^{-\lambda |\x|} - e^{-\lambda} \big\}\,\d\sigma(\lambda),
\end{equation}
where $\sigma$ is a suitable nonnegative Borel measure on $(0,\infty)$. The class \eqref{Exp_Sub} includes examples such as $\x \mapsto -\log |\x|$ and $\x \mapsto \Gamma(\alpha +1) \big\{|\x|^{-\alpha-1}-1\big\}$ for $-2 < \alpha < 2\nu +1$. The primary objective of this paper is to extend the framework of \cite{CL3}, providing a complete solution of our extremal problem \eqref{Intro_EP1} - \eqref{Intro_EP2} for the Gaussian $\x \mapsto e^{- \pi \lambda |\x|^2}$, where $\lambda >0$ (that is not included in the class \eqref{Exp_Sub}), and generalizing the construction to a wide class of radial functions that {\it strictly contains} \eqref{Exp_Sub}. Previously, the special case $\nu = -1/2$ and $N=1$  of the Gaussian subordination framework was obtained in \cite{CLV} with an interpolation method that does not extend to arbitrary $\nu$ and $N$. The results here are based on a different approach described in Section \ref{Int_type_zero} which is substantially simpler even for $\nu = -1/2$ and $N=1$.

The radial structure is a crucial element in this work, in the sense that it allows us to connect the multidimensional problem to a weighted one-dimensional problem, which we solve using the rich theory of de Branges spaces of entire functions \cite{B}.

\subsection{De Branges spaces} Before stating our main results, we must recall the basic features of the theory of de Branges spaces \cite{B}. Throughout the text we write
$$\U = \{z \in \C;\ \im(z) >0\}$$ 
for the open upper half-plane. An analytic function $F: \U \to \C$ is said to have {\it bounded type} if it can be written as a quotient of two bounded and analytic functions in $\U$. Functions of bounded type in $\U$ admit a Nevanlinna factorization \cite[Theorems 9 and 10]{B}, from which we draw the number
\begin{equation*}
v(F) := \limsup_{y \to \infty} \, y^{-1}\log|F(iy)| <\infty,
\end{equation*}
called the {\it mean type} of $F$. It also follows from \cite[Theorems 9 and 10]{B} that the mean type verifies 
$$v(FG) = v(F) + v(G)$$ 
whenever $F$ and $G$ have bounded type in $\U$. 

\smallskip

If $E: \C \to \C$ is an entire function, we let $E^*:\C \to \C$ be the entire function given by $E^*(z) = \ov{E(\ov{z})}$. We say that $E$ is a {\it Hermite-Biehler} function if 
\begin{equation}\label{Intro_HB_cond}
|E^*(z)| < |E(z)|
\end{equation}
for all $z \in \U$, and we say that $E$ is {\it real entire} if its restriction to $\R$ is real-valued. If $E: \C \to \C$ is a Hermite-Biehler function, we consider the vector space $\H(E)$ of entire functions $F: \C \to \C$ such that 
\begin{equation*}
\|F\|_E^2 := \int_{-\infty}^\infty |F(x)|^{2} \, |E(x)|^{-2} \, \dx <\infty\,,
\end{equation*}
and such that $F/E$ and $F^*/E$ have both bounded type in $\U$ and nonpositive mean type. The space $\H(E)$ is in fact a {\it reproducing kernel Hilbert space}, with inner product given by
\begin{equation*}
\langle F, G \rangle_E := \int_{-\infty}^\infty F(x) \,\overline{G(x)} \, |E(x)|^{-2} \, \dx.
\end{equation*}
This means that for each $w \in \C$ the evaluation functional $F \mapsto F(w)$ is continuous, and thus there exists a function $z \mapsto K(w,z) \in \H(E)$ such that 
\begin{equation}\label{Intro_rep_pro}
F(w) = \langle F, K(w,\cdot) \rangle_E
\end{equation}
for all $F \in \H(E)$. The function $K(w,z)$ is called the {\it reproducing kernel} of $\H(E)$. 

\smallskip

We denote by $A$ and $B$ the following companion functions
\begin{equation}\label{Intro_def_A_B}
A(z) := \frac12 \big\{E(z) + E^*(z)\big\} \ \ \ {\rm and}  \ \ \ B(z) := \frac{i}{2}\big\{E(z) - E^*(z)\big\}.
\end{equation}
Note that $A$ and $B$ are real entire functions such that $E(z) = A(z) -iB(z)$. The reproducing kernel is then given by \cite[Theorem 19]{B}
\begin{align}\label{Intro_Def_K_w_z}
K(w,z) &= \frac{B(z)A(\ov{w}) - A(z)B(\ov{w})}{\pi (z - \ov{w})} = \frac{E(z)E^*(\ov{w}) - E^*(z)E(\ov{w})}{2\pi i (\ov{w}-z)},
\end{align}
and when $z = \ov{w}$ we have
\begin{equation}\label{Intro_Def_K}
K(\ov{z}, z) = \frac{B'(z)A(z) - A'(z)B(z)}{\pi}.
\end{equation}
From \eqref{Intro_rep_pro} we obtain
\begin{align*}
0 \leq \|K(w, \cdot)\|_E^2 = \langle K(w, \cdot), K(w, \cdot) \rangle_E = K(w,w),
\end{align*}
and we notice that $K(w,w)=0$ if and only if $w \in\R$ and $E(w) = 0$ (see for instance \cite[Lemma 11]{HV}).

\subsection{Approximating the Gaussian in de Branges spaces}

Throughout the paper we let $E$ be a Hermite-Biehler function satisfying the following properties:
\begin{enumerate}
\item[(P1)] $E$ has bounded type in $\U$;
\smallskip
\item[(P2)] $E$ has no real zeros; 
\smallskip
\item[(P3)] $z \mapsto E(iz)$ is a real entire function (or, equivalently, $E^*(z) = E(-z)$ for all $z \in \C$).
\smallskip
\item[(P4)] $A, B \notin \H(E)$. 
\end{enumerate} 
\smallskip
The fact that $E$ is a Hermite-Biehler function that satisfies (P1) implies that $E$ has exponential type and $\tau(E) = v(E)$. This is a consequence of Krein's theorem (see \cite{K} or \cite[Lemma 9]{HV}). Also, from the Hermite-Biehler condition \eqref{Intro_HB_cond}, we note that the companion functions $A$ and $B$ only have real zeros, and moreover these zeros are all simple due to \eqref{Intro_Def_K} and property (P2). We are now able to state our first result, the solution to a weighted one-dimensional extremal problem for the Gaussian, with respect to a general de Branges metric.

\begin{theorem}\label{Thm1}
Let $\lambda>0$. Let $E$ be a Hermite-Biehler function satisfying properties {\rm (P1) \!-\! (P4)} above. Assume also that
\begin{equation*}
\int_{-\infty}^{\infty} e^{- \pi \lambda |x|^2}\, |E(x)|^{-2}\, \dx < \infty.
\end{equation*}
The following properties hold:
\smallskip
\begin{enumerate}
\item[(i)] If $L:\C \to \C$ is an entire function of exponential type at most $2 \tau(E)$ such that $L(x) \leq e^{- \pi \lambda |x|^2}$ for all $x \in \R$ then 
\begin{equation}\label{Intro-L-integral}
\int_{-\infty}^{\infty} L(x) \,|E(x)|^{-2}\,\dx \leq \sum_{A(\xi)=0} \frac{e^{- \pi \lambda |\xi|^2}}{K(\xi,\xi)},
\end{equation}
where the sum on the right-hand side of \eqref{Intro-L-integral} is finite. Moreover, there exists an entire function $z\mapsto L(A^2, \lambda, z)$ of exponential type at most $2 \tau(E)$ such that $L(A^2, \lambda, x) \leq e^{-\pi \lambda  |x|^2}$ for all $x \in \R$ and equality in \eqref{Intro-L-integral} holds.
\medskip
\item[(ii)] If $M:\C \to \C$ is an entire function of exponential type at most $2 \tau(E)$ such that $M(x) \geq e^{- \pi \lambda |x|^2}$ for all $x \in \R$ then 
\begin{equation}\label{Intro-M-integral}
\int_{-\infty}^{\infty} M(x) \,|E(x)|^{-2}\,\dx \geq \sum_{B(\xi)=0} \frac{e^{-\pi \lambda |\xi|^2}}{K(\xi,\xi)},
\end{equation}
where the sum on the right-hand side of \eqref{Intro-M-integral} is finite. Moreover, there exists an entire function $z\mapsto M(B^2, \lambda, z)$ of exponential type at most $2 \tau(E)$ such that $M(B^2, \lambda, x) \geq e^{-\pi \lambda |x|^2}$ for all $x \in \R$ and equality in \eqref{Intro-M-integral} holds.
\end{enumerate}
\end{theorem}

There are several important ingredients that come together in the proof of Theorem \ref{Thm1}. Among these we highlight the following: (a) the construction of suitable entire functions of exponential type that interpolate the Gaussian at the zeros of a given Laguerre-P\'{o}lya function; (b) the decomposition of an $L^1(\R, |E(x)|^{-2}\,\dx)$-function of exponential type $2\tau(E)$ as a difference of squares of functions in $\H(E)$; (c) the use of Plancherel's identity on the Hilbert space $\H(E)$. Since $A$ and $B$ are defined in terms of $E$, the extremal functions of Theorem \ref{Thm1} depend implicitly on $E$, and the particular choice of notation  
$z\mapsto L(A^2, \lambda, z)$ and $z\mapsto M(B^2, \lambda, z)$ will become clear in Section \ref{Int_type_zero}.

\subsection{Approximating the Gaussian in Euclidean spaces} We shall use our Theorem \ref{Thm1} to obtain the solution of the multidimensional Beurling-Selberg extremal problem \eqref{Intro_EP1} - \eqref{Intro_EP2} for the Gaussian. As in \cite{CL3,HV}, the main insight here is the specialization of Theorem \ref{Thm1} to a particular family of homogeneous de Branges spaces, suitable to treat the power weights $|\x|^{2\nu + 2 - N}\,\d\x$. We start by briefly reviewing the basic properties of these spaces. 

\smallskip

For $\nu > -1$ let $A_{\nu}:\C \to \C$ and $B_{\nu}:\C \to \C$ be real entire functions defined by 
\begin{equation}\label{Intro_A_nu}
A_{\nu}(z) = \sum_{n=0}^{\infty} \frac{(-1)^n \big(\tfrac12 z\big)^{2n}}{n!(\nu +1)(\nu +2)\ldots(\nu+n)}
\end{equation}
and
\begin{equation}\label{Intro_B_nu}
B_{\nu}(z) = \sum_{n=0}^{\infty} \frac{(-1)^n \big(\tfrac12 z\big)^{2n+1}}{n!(\nu +1)(\nu +2)\ldots(\nu+n+1)}.
\end{equation}
These functions are related to the classical Bessel functions by the identities
\begin{align}
A_{\nu}(z) &= \Gamma(\nu +1) \left(\tfrac12 z\right)^{-\nu} J_{\nu}(z),\label{Intro_Bessel1}\\
B_{\nu}(z) & = \Gamma(\nu +1) \left(\tfrac12 z\right)^{-\nu} J_{\nu+1}(z).\label{Intro_Bessel2}
\end{align}
It follows that both $A_{\nu}$ and $B_{\nu}$ have only real, simple zeros and have no common zeros (note that in the simplest case $\nu = -1/2$ we have $A_{-1/2}(z) = \cos z$ and $B_{-1/2}(z) = \sin z$). Moreover, they satisfy the system of differential equations 
\begin{equation}\label{Intro_Diff_Eqs}
A_{\nu}'(z) = - B_{\nu}(z) \ \ \ {\rm and} \ \ \ B_{\nu}'(z) = A_{\nu}(z) - (2\nu +1)z^{-1}B_{\nu}(z).
\end{equation}
The entire function 
$$E_{\nu}(z) := A_{\nu}(z) - iB_{\nu}(z)$$
is a Hermite-Biehler function that satisfies properties (P1) \!-\! (P4) listed above, with $v(E_{\nu}) = \tau(E_{\nu}) = 1$ (we shall verify these assertions later), and we denote by $K_{\nu}$ the reproducing kernel of the de Branges space $\mc{H}(E_{\nu})$.

\smallskip

For $\delta >0$ and $N \in \N$ we define $\E_{\delta}^{N}$ to be the set of all entire functions $F:\C^N \to \C$ of exponential type at most $\delta$. Given a real-valued function $\mc{F}:\R^N \to \R$ we write
\begin{align*}
\E_{\delta}^{N-} (\mc{F}) &:= \big\{\mc{L}\in \E_{\delta}^N;\ \mc{L}(\x) \leq \mc{F}(\x), \forall \x \in \R^N \big\},\\[0.5em]
\E_{\delta}^{N+}(\mc{F}) &:= \big\{\mc{M} \in \E_{\delta}^N;\ \mc{M}(\x) \geq \mc{F}(\x), \forall \x \in \R^N \big\}.\\[-1em]
\end{align*}

\begin{theorem}\label{Thm2}
Let $\nu>-1$, $\lambda >0$ and $\mc{F}_{\lambda}(\x) = e^{- \pi \lambda |\x|^2}$. Let
\begin{align*}
U_{\nu}^{N-}(\delta, \lambda)  & = \inf\left\{ \int_{\R^N} \left\{e^{- \pi \lambda |\x|^2} - \mc{L}(\x)\right\}\,|\x|^{2\nu+2 - N}\,\dxx; \  \mc{L} \in \E_{\delta}^{N-} (\mc{F}_{\lambda}) \right\},\\
U_{\nu}^{N+}(\delta, \lambda)  & = \inf\left\{ \int_{\R^N} \left\{\mc{M}(\x) - e^{-  \pi \lambda |\x|^2}\right\}\,|\x|^{2\nu+2 - N}\,\dxx; \  \mc{M} \in \E_{\delta}^{N+} (\mc{F}_{\lambda}) \right\}.
\end{align*}
The following properties hold:
\begin{enumerate}
\item[(i)] If $\kappa > 0$ then $U_{\nu}^{N\pm}(\delta, \lambda) = \kappa^{2\nu +2} \,U_{\nu}^{N\pm}(\kappa\delta, \kappa^2\lambda)$.
\smallskip
\item[(ii)] For any $\delta >0$ we have
\begin{equation*}\label{Intro_rel_N_1}
U_{\nu}^{N\pm}(\delta, \lambda) = \tfrac12\, \omega_{N-1}\, U_{\nu}^{1\pm}(\delta, \lambda),
\end{equation*}
where $\omega_{N-1} = 2 \pi^{N/2} \,\Gamma(N/2)^{-1}$ is the surface area of the unit sphere in $\R^N$.
\smallskip
\item[(iii)] When $N=1$ and $\delta = 2$ we have
\begin{align}
U_{\nu}^{1-}(2, \lambda)&= \frac{ \,\Gamma(\nu +1)}{ (\pi \lambda)^{\nu +1}}- \sum_{A_{\nu}(\xi)=0}\frac{e^{- \pi \lambda |\xi|^2}}{c_{\nu}\,K_{\nu}(\xi, \xi)}, \label{Intro_value_min}\\
U_{\nu}^{1+}(2, \lambda)& = \sum_{B_{\nu}(\xi)=0}\frac{e^{-\pi \lambda|\xi|^2}}{c_{\nu}\,K_{\nu}(\xi, \xi)} - \frac{ \,\Gamma(\nu +1)}{(\pi \lambda)^{\nu +1}},\label{Intro_value_max}
\end{align}
where $c_{\nu} = \pi\, 2^{-2\nu -1}\, \Gamma(\nu +1)^{-2}.$

\smallskip

\item[(iv)] There exists a pair of real entire functions of $N$ complex variables, $\z \mapsto \mc{L}_{\nu,N}(\delta, \lambda, \z) \in \E_{\delta}^{N-} (\mc{F}_{\lambda}) $ and $\z \mapsto \mc{M}_{\nu,N}(\delta, \lambda, \z) \in \E_{\delta}^{N+} (\mc{F}_{\lambda})$, which are radial when restricted to $\R^N$ and verify 
\begin{equation*}
U_{\nu}^{N-}(\delta, \lambda) = \int_{\R^N} \left\{e^{-\pi \lambda|\x|^2} - \mc{L}_{\nu,N}(\delta, \lambda, \x) \right\}\,|\x|^{2\nu+2-N}\,\dxx
\end{equation*}
and
\begin{equation*}
U_{\nu}^{N+}(\delta, \lambda) = \int_{\R^N} \left\{\mc{M}_{\nu,N}(\delta, \lambda, \x) - e^{-\pi \lambda|\x|^2}  \right\}\,|\x|^{2\nu+2-N}\,\d\x.
\end{equation*}
\end{enumerate}
\end{theorem}

In order to extend the construction of Theorem \ref{Thm2} to a class of radial functions, it is important to understand the asymptotics of \eqref{Intro_value_min} and \eqref{Intro_value_max} as $\lambda \to 0$ and as $\lambda \to +\infty$. From \eqref{Intro_Def_K_w_z}, \eqref{Intro_Bessel1}, \eqref{Intro_Bessel2} and \eqref{Intro_Diff_Eqs} we have, for $\xi >0$, 
\begin{align}\label{Intro_asymp1}
K_{\nu}(\xi, \xi) = 2^{-1}\, c_{\nu}^{-1}\, \xi^{-2\nu -1} \Big\{ \xi J_{\nu}(\xi)^2 + \xi J_{\nu+1}(\xi)^2 - \!(2\nu+1)\,J_{\nu}(\xi)\,J_{\nu+1}(\xi)\Big\}.
\end{align}
Using the well-known asymptotic for Bessel functions
\begin{equation}\label{Asymptotic_Bessel_functions}
J_{\nu}(\xi) =  \sqrt{\frac{2}{\pi \xi}} \,  \cos\left(\xi - \tfrac{\pi \nu}{2} - \tfrac{\pi}{4}\right) + O_{\nu}\big(\xi^{-{3/2}}\big)
\end{equation}
as $\xi \to \infty$, we have
\begin{equation}\label{Intro_asymp2}
\lim_{\xi \to \infty} \Big\{ \xi J_{\nu}(\xi)^2 + \xi J_{\nu+1}(\xi)^2 - \!(2\nu+1)\,J_{\nu}(\xi)\,J_{\nu+1}(\xi)\Big\} = \frac{2}{\pi}.
\end{equation}
It then follows that
\begin{align}\label{Intro_Asymp_K}
K_{\nu}(\xi, \xi) \sim \pi^{-1} \, c_{\nu}^{-1}\, \xi^{-2\nu -1}
\end{align}
as $\xi \to \infty$. Also, notice by \eqref{Intro_Bessel1} and \eqref{Intro_Bessel2} that the positive zeros of $A_{\nu}$ and $B_{\nu}$ are just the positive zeros of $J_{\nu}$ and $J_{\nu+1}$, and it is well-known that, for large $m$, the Bessel function $J_{\nu}$ has exactly $m$ zeros in the interval $(0, m\pi + \frac{\pi \nu}{2} + \frac{\pi}{4}]$ (see \cite[Section 15.4]{W}). Denoting by $j_{\nu,n}$ the $n$-th positive zero of $J_\nu$ (ordered by size), it follows from \eqref {Intro_value_min} and \eqref{Intro_Asymp_K} that
\begin{align*}
\big|\,U_{\nu}^{1-}(2, \lambda) \big| & \ll_{\nu} \frac{1}{\lambda^{\nu +1}} + \sum_{n=1}^{\infty} n^{2\nu+1}\,e^{-\pi \lambda j_{\nu,n}^2} \\
& \ll_{\nu} \frac{1}{\lambda^{\nu +1}}\,
\end{align*}
as $\lambda \to \infty$ (we use Vinogradov's notation $f \ll g$, or $f = O(g)$, to mean that $|f|$ is less than or equal to a constant times $|g|$, and indicate the parameters of dependence of the constant in the subscript). Analogously, from \eqref {Intro_value_max} and \eqref{Intro_Asymp_K} we find
\begin{align*}
\big|U_{\nu}^{1+}(2,\lambda) \big|& \ll_{\nu} \frac{1}{c_{\nu} K_{\nu}(0,0)} + \sum_{n=1}^{\infty} n^{2\nu+1}\,e^{-\pi \lambda j_{\nu+1,n}^2}  +  \frac{1}{\lambda^{\nu +1}}\\
&  \ll_{\nu} 1
\end{align*}
as $\lambda \to \infty$. To complete the asymptotic description we shall prove in the beginning of Section \ref{Sec_Gaussian_Sub} that 
\begin{equation}\label{Intro_open_asymp}
\big|U_{\nu}^{1\pm}(2,\lambda) \big| \ll_{\nu,k}  \lambda^k
\end{equation}
as $\lambda \to 0$, for every $k\in\N$. It is precisely this fact that makes the Gaussian subordination framework more robust (when compared to the analogous framework generated by exponential $\x \mapsto e^{-\lambda |\x|}$ in \cite{CL3}, that verifies \eqref{Intro_open_asymp} only for $k=1$). A somewhat surprising feature is the fact that the solution for the exponential is an important step towards the solution for the Gaussian, as we shall see in Section \ref{Int_type_zero}.

\subsection{Gaussian subordination} We now proceed with the extension of Theorem \ref{Thm2} to a class of radial functions. The way we define our class of functions below might look slightly odd at first, but it will be clear later on that this is indeed the most general way to proceed. This is essentially motivated by the fact that, in order to make the best use of the strong decay at the origin given by \eqref{Intro_open_asymp}, we are led to work on the Fourier transform side. In general, this framework covers functions that are not integrable. This requires us to work with Fourier transforms in the distributional sense.

\smallskip

We consider two classes of nonnegative Borel measures $\mu$ on $(0,\infty)$ for a given parameter $\nu>-1$. For the minorant problem we consider the class of measures $\mu$ satisfying the condition
\begin{equation}\label{Intro_mu_1}
\int_{0}^{\infty} \frac{\lambda^k}{1 + \lambda^{\nu + k+1}}\,\dmu < \infty
\end{equation}
for some $k\in\N$, whereas for the majorant problem we consider the class of measures $\mu$ satisfying the more restrictive condition
\begin{equation}\label{Intro_mu_2}
\int_{0}^{\infty} \frac{\lambda^k}{1 + \lambda^k}\,\dmu < \infty
\end{equation}
for some $k\in\N$. Let $g_\mu:\R\to \R \cup \{\infty\}$ be a function satisfying the following properties:

\begin{enumerate}
\item[(Q1)] The function $g_\mu$ is even.

\smallskip

\item[(Q2)] For the majorant problem, the function $g_\mu$ is finite-valued and continuous on $\R$. For the minorant problem, the function $g_\mu$ is finite-valued and continuous on $\R\setminus \{0\}$ and $g_{\mu}(0) = \limsup_{x \to 0} g_{\mu}(x)$.

\smallskip

\item[(Q3)] For each $N \in \N$, let $\G_{\mu,N}:\R^N\to \R$ be the radial extension of $g_{\mu}$ to $\R^N$, i.e.
\begin{align}\label{re-g}
\G_{\mu,N}(\x) = g_\mu(|\x|).
\end{align}
If $N_\nu = \lceil 2 \nu + 2 \rceil$ (the smallest integer greater than or equal to $2\nu+2$), then $\G_{\mu,N_\nu}$ is a tempered distribution on $\R^{N_\nu}$.

\smallskip

\item[(Q4)] For any Schwartz function $\varphi:\R^{N_{\nu}} \to \C$ that is zero in an open neighborhood of the origin we have
\begin{equation}\label{Intro_Def_via_FT}
\int_{\R^{N_\nu}} \G_{\mu,N_{\nu}}(\x)\,\widehat{\varphi}(\x)\,\d\x = \int_{\R^{N_\nu}} \left(\int_0^\infty \lambda^{-N_\nu/2} \,e^{-\pi|\t|^2/\lambda} \,\d\mu(\lambda)\right) \,\varphi(\t)\,\d\t,
\end{equation}
which, informally speaking, means that we require the Fourier transform of  $\G_{\mu,N_{\nu}}$ to equal the expression in parentheses in \eqref{Intro_Def_via_FT} for $\t \neq0$.

\smallskip
\end{enumerate}

For $\kappa >0$ we define a new measure $\mu_{\kappa}$ on the Borel subsets $X \subseteq (0,\infty)$ by putting
\begin{align}\label{def_mu_kappa}
\mu_{\kappa}(X) := \mu(\kappa X),
\end{align}
where 
$$\kappa X = \{ \kappa x; \ x \in X\}.$$
Note that $\mu_{\kappa}$ satisfies \eqref{Intro_mu_1} or \eqref{Intro_mu_2} whenever $\mu$ does. We are now in position to state our next result.

\begin{theorem}\label{Thm3}
Let $\mu$ be a nonnegative Borel measure on $(0,\infty)$ satisfying \eqref{Intro_mu_1} for the minorant problem or \eqref{Intro_mu_2} for the majorant problem. Let $g_\mu:\R \to \R \cup \{\infty\}$ be a function satisfying properties {\rm (Q1) \!-\! (Q4)} above and let $\mc{G}_{\mu,N}:\R^N \to \R$ be defined by \eqref{re-g}. For $N\in \N$ write
\begin{align*}
U_{\nu}^{N-}(\delta, \mu)  & = \inf\left\{ \int_{\R^N} \Big\{\mc{G}_{\mu,N}(\x) - \mc{L}(\x)\Big\}\,|\x|^{2\nu+2-N}\,\d\x; \  \mc{L} \in \E_{\delta}^{N-} (\mc{G}_{\mu,N}) \right\},\\
U_{\nu}^{N+}(\delta, \mu) & = \inf\left\{ \int_{\R^N} \Big\{\mc{M}(\x) - \mc{G}_{\mu,N}(\x)\Big\}\,|\x|^{2\nu+2-N}\,\d\x; \  \mc{M} \in \E_{\delta}^{N+}(\mc{G}_{\mu,N})  \right\}.
\end{align*}
The following properties hold:
\smallskip
\begin{enumerate}
\item[(i)] If $0 < \kappa$ then $U_{\nu}^{N\pm}(\delta, \mu) = \kappa^{2\nu +2} \,U_{\nu}^{N\pm}(\kappa\delta, \mu_{\kappa^{-2}})$.
\smallskip
\item[(ii)] For any $\delta >0$ we have
\begin{equation}\label{Intro_final_answer_int_parameter}
U_{\nu}^{N\pm}(\delta, \mu)= \int_0^{\infty} U_{\nu}^{N\pm}(\delta, \lambda)\,\dmu = \tfrac12 \, \omega_{N-1}\int_0^{\infty} U_{\nu}^{1\pm}(\delta, \lambda)\,\dmu.
\end{equation}
where $\omega_{N-1} = 2 \pi^{N/2} \,\Gamma(N/2)^{-1}$ is the surface area of the unit sphere in $\R^N$.

\smallskip

\item[(iii)] There exists a pair of real entire functions of $N$ complex variables, $\z \mapsto \mc{L}_{\nu,N}(\delta, \mu, \z) \in \E_{\delta}^{N-} (\mc{G}_{\mu,N}) $ and $\z \mapsto \mc{M}_{\nu,N}(\delta, \mu, \z) \in \E_{\delta}^{N+} (\mc{G}_{\mu,N})$, which are radial when restricted to $\R^N$ and verify
\begin{equation*}
U_{\nu}^{N-}(\delta, \mu) = \int_{\R^N} \big\{\mc{G}_{\mu,N}(\x) - \mc{L}_{\nu,N}(\delta, \mu, \x) \big\}\,|\x|^{2\nu+2-N}\,\dxx
\end{equation*}
and
\begin{equation*}
U_{\nu}^{N+}(\delta, \mu) = \int_{\R^N} \big\{\mc{M}_{\nu,N}(\delta, \mu, \x) - \mc{G}_{\mu,N}(\x)  \big\}\,|\x|^{2\nu+2-N}\,\dxx.
\end{equation*}
\end{enumerate}
\end{theorem}

\noindent {\sc Remark:} When restricted to the class of entire functions of $N$ complex variables {\it that are radial on $\R^N$}, the extremal functions described in Theorem \ref{Thm2} (iv) and Theorem \ref{Thm3} (iii) {\it are unique}. This follows from the recent work of F. Gon\c{c}alves \cite[Section 4.2]{Gon} on interpolation formulas with derivatives in de Branges spaces. On the other hand, uniqueness of the extremal functions described in Theorem \ref{Thm1} is only known under additional conditions on $E$ (see \cite[Theorem 1]{Gon}).

\smallskip

The idea of solving the extremal problem for a base case with a free parameter $\lambda >0$, and then integrating the parameter against a suitable measure $\mu$ is reminiscent of the work of Graham and Vaaler \cite{GV}. As the final result \eqref{Intro_final_answer_int_parameter} suggests, a necessary condition to be imposed on the measure $\mu$ is the $\mu$-integrability of $U_{\nu}^{1\pm}(\delta, \lambda)$ considered as a function of $\lambda$. In order to generate the largest possible class, one wants to make this necessary condition also sufficient, and that is precisely the reason  why we define our objects via the Fourier transform (since the Fourier transform of the Gaussian in \eqref{Intro_Def_via_FT} has an exponential decay to zero as $\lambda \to 0$ for each fixed $\t \neq 0$). The passage to dimension $N_\nu$ in \eqref{Intro_Def_via_FT} guarantees the proper decay as $\lambda \to \infty$. In Section \ref{App} we describe a variety of examples that can be obtained via our Gaussian subordination method. In particular, this new framework strictly contains the class \eqref{Exp_Sub} of our precursor \cite{CL3}.

\smallskip

The one-dimensional Gaussian subordination framework \cite{CL, CLV} provided the solution of the Beurling-Selberg extremal problems for previously inaccessible functions such as $f_{a,b}(x) = \log\big((x^2 + b^2)/(x^2 + a^2)\big)$, where $0 \leq a < b$\,; $g(x) = \arctan(1/x) - x/(1+x^2)$ and $h(x) = 1 - x\arctan(1/x)$. These were used to improve the existing bounds for the modulus and the argument of the Riemann zeta-function on the critical strip under RH (also for $L$-functions under GRH) \cite{CC, CCM, CCM2}. We expect that the extension of this framework to a multidimensional setting presented here might reveal other interesting applications.

%
%

\section{Entire functions with prescribed nodes} \label{Int_type_zero}

\subsection{Laguerre-P\'{o}lya functions} We recall that a Laguerre-P\'{o}lya function $F$ is a real entire function with Hadamard factorization given by\begin{equation}\label{hadamardproduct}
F(z)=\frac{F^{(r)}(0)}{r!}\,z^r\, e^{-az^2+bz}\,\prod_{j=1}^\infty\Big(1-\frac{z}{\xi_j}\Big)e^{z/\xi_j},
\end{equation}
where $r\in\Z^{+}$, $a, b, \xi_j \in \R$, with $a \geq 0$, $\xi_j \neq 0$ and $\sum_{j=1}^\infty \xi_j^{-2}<\infty$ (the product may be finite or empty). We say that $F$ has degree $\mc{N} = \mc{N} (F)$, with $0 \leq \mc{N}< \infty$, if $a=0$ in \eqref{hadamardproduct} and $F$ has exactly $\mc{N}$ zeros counted with multiplicity. Otherwise we set $\mc{N}(F) = \infty$.

\smallskip

The important property of Laguerre-P\'{o}lya functions for our purposes is that we can write their inverses as Laplace transforms on vertical strips. In fact, if $(\tau_1, \tau_2) \subset  \R$ is an open interval not containing any zeros of $F$ (we allow $\tau_1,\tau_2\in \{\pm\infty\}$) and $c \in (\tau_1, \tau_2)$, we have 
\begin{equation}\label{gc-trafo}
\frac{1}{F(z)} = \int_{-\infty}^\infty g_c(t)\, e^{-zt}\, \,\dt
\end{equation}
for $\tau_1 < \re(z) < \tau_2$, where 
\begin{equation}\label{gc-def}
g_c(t) = \frac{1}{2\pi i }\int_{c-i\infty}^{c+i\infty} \frac{e^{st}}{F(s)} \,\ds.
\end{equation}
The integral \eqref{gc-def} is absolutely convergent if $\mc{N}(F) \geq 2$, and is understood as a Cauchy principal value if $\mc{N}(F) =1$. If $\mc{N}(F) =0$, \eqref{gc-trafo} holds with $g_c$ being an appropriate Dirac delta distribution (though we will not be interested in this case here). A simple application of the residue theorem gives us that $g_c = g_d$ if $c,d \in (\tau_1, \tau_2)$. We briefly review below the main qualitative properties of the frequency functions $g_c$ needed for this work. We refer the reader to \cite[Chapters II to V]{HW} for a detailed account of this theory.

\begin{lemma}\label{g-sign-changes} 
Let $F$ be a Laguerre-P\'olya function with at least one zero and let $g_c$ be defined by \eqref{gc-def}, where $c \in \R$ and $F(c) \neq 0$. The following propositions hold:
\begin{enumerate} 
\item[(i)] The function $g_c$ is real-valued and of one sign, and its sign is equal to the sign of $F(c)$.

\smallskip

\item [(ii)] If $F$ has degree $\mc{N} \geq 2$, then  $g_c \in C^{\mc{N}-2}(\R)$.
 
 \smallskip
 
\item[(iii)]  If $(\tau_1,\tau_2) \subset \R$ is the largest open interval containing no zeros of $F$ such that $c \in (\tau_1,\tau_2)$, then for any $\tau \in (\tau_1,\tau_2)$ we have the estimate
\begin{align}\label{Lem4_growth}
\big|g_c(t)\big| \ll_{\tau} e^{\tau t} \ \ \  \forall \,t \in \R.
\end{align}

\end{enumerate}
\end{lemma}

\begin{proof} Part (i) follows from \cite[Chapter IV, Theorem 5.1]{HW}. For part (ii), observe that we can differentiate under the integral sign in \eqref{gc-def} since 
\begin{equation}\label{Sec2_Est_g_c}
1/|F(c+iy)| = O(|y|^{-k})
\end{equation} 
as $|y| \to \infty$ if $\mathcal{N}(F) \geq k$ (for details see \cite[Chapter II, Theorem 6.3 and Chapter III, Theorem 5.3]{HW}).  For part (iii), observe that $g_c = g_{\tau}$ and estimate the integral in \eqref{gc-def} trivially using \eqref{Sec2_Est_g_c} (in the case $\mathcal{N}(F) \geq 2$). If $\mathcal{N}(F)= 1$, we can see it directly as in \eqref{Sec2_case_N_1} below.
\end{proof}

We use the notation $g_{c-}$ for the function $g_{c-\varepsilon}$ in \eqref{gc-def}, with $\varepsilon$ so small that there is no zero of $F$ in the interval $(c-\varepsilon,c)$, and analogously for $g_{c+}$. In this paper we are especially interested in the case when the nonzero roots $\{\xi_j\}_{j=1}^{\infty}$ of $F$ verify
\begin{align}\label{zero-conv}
 \xi_j>0\text{ for all $j$ \   and \   }\sum_{j=1}^\infty \xi_j^{-1} <\infty.
\end{align}
The next lemma presents an essential property of the Laguerre-P\'{o}lya functions satisfying \eqref{zero-conv}.

\begin{lemma}\label{Sec2_lem1} 
Let $k\in\{0,1\}$. Let $F$ be a Laguerre-P\'{o}lya function of the form
\begin{align}\label{F-absconv}
F(z) = z^k \prod_{j=1}^\infty \left( 1- \frac{z}{\xi_j}\right),
\end{align}
where the points $\{\xi_j\}_{j=1}^{\infty}$ satisfy \eqref{zero-conv}. Then the function  $g_{0-}$ defined in \eqref{gc-def} satisfies $g_{0-}(t) =0$ for all $t>0$. In particular, \eqref{gc-trafo} becomes, for $\re(z)<0$,
\begin{align}\label{gc-trafo-special}
\frac{1}{F(z)} = \int_{-\infty}^0 g_{0-}(t)\, e^{-zt}\, \dt.
\end{align}
\end{lemma}

\begin{proof} Let us consider first the case $k=0$. Assume that $0 < \xi_1 \leq \xi_2 \leq \xi_3 \leq \ldots$ and define
$$F_n(z) =  \prod_{j=1}^n \left( 1- \frac{z}{\xi_j}\right).$$
For $a>0$, we define the function $\gamma_{a-}$ by 
\begin{align*}
\gamma_{a-}(t) &=  \begin{cases} \displaystyle a e^{at}, &\text{ if }t<0; \\ \displaystyle 0, & \text{ if }t \ge 0.
\end{cases} 
\end{align*}
We note that 
\begin{equation}\label{Sec2_case_N_1}
\frac{1}{1- \frac{z}{a}} = \int_{-\infty}^\infty \gamma_{a-}(t)\,e^{-zt} \,\dt 
\end{equation}
in the half-plane $\re(z)<a$. It is then straightforward to check that
\[
\frac{1}{F_n(z)} =\int_{-\infty}^\infty (\gamma_{\xi_1-} * \ldots * \gamma_{\xi_n-})(t)\,e^{-zt}\,  \dt
\]
in the half-plane $\re(z)<\xi_1$. Note that all functions $t \mapsto \gamma_{a-}(t)$, for $a>0$, belong to $L^1(\R) \cap L^{\infty}(\R)$, hence the convolution $ t \mapsto (\gamma_{\xi_1-} * \ldots * \gamma_{\xi_n-})(t)$ is a well-defined continuous and integrable function for $n\geq 2$. Since $1/F_n$ converges uniformly to $1/F$ on the line $\re(z) =0$ and $\xi_1>0$, it follows from \eqref{gc-def} that $ (\gamma_{\xi_1-} * \ldots * \gamma_{\xi_n-})$ converges pointwise to $ g_0$ as $n\to \infty$. In particular,  $g_{0}(t)=0$ for $t>0$. Since $g_0 = g_{0-}$, the result is shown if $k=0$.  If $k=1$, the $n$-fold convolution is convolved with the function that equals $-1$ for $t<0$ and $0$ for $t\ge 0$. We use in this case that $1/F_n$ converges uniformly to $1/F$ on the line $\re(z) = -\varepsilon$. The calculations are analogous.
\end{proof}

\subsection{Interpolating the exponential} Let $F$ be a Laguerre-P\'{o}lya function given by \eqref{F-absconv} where the points $\{\xi_j\}_{j=1}^{\infty}$ satisfy \eqref{zero-conv}. Let $g = g_{0-}$ as in \eqref{gc-def} and define, for $\lambda >0$, the functions
\begin{align}\label{A-rep}
\begin{split}
\A(F,\lambda,z) &= e^{- \lambda z} - F(z) \int_0^\lambda g(w- \lambda)\, e^{- zw}\, \dw, \ \ \ {\rm for} \ z \in \C,\\
\A_1(F,\lambda,z) &= F(z) \int_{-\infty}^0 g(w- \lambda) \,e^{- zw}\, \dw, \ \ \ {\rm for} \ \re(z)< 0.
\end{split}
\end{align}
Since  the smallest zero of $F$ is nonnegative,  Lemma \ref{g-sign-changes} (iii) implies that $z\mapsto \A(F,\lambda,z)$ is entire and $z\mapsto \A_1(F,\lambda,z)$ is analytic in the half-plane $\re(z)<0$. Also, it follows directly from \eqref{gc-trafo-special} that 
\begin{equation}\label{Sec2_A_equal_A_1}
\A(F,\lambda,z) = \A_1(F,\lambda,z)
\end{equation}
in $\re(z)<0$.

\smallskip

We now define, for $z\in\C$ and $\lambda >0$,  the functions
\begin{align*}
H(z) &=\begin{cases}
1,\text{ if }\re (z)\ge -1/4;\\
0,\text{ if }\re (z)< -1/4;
\end{cases}\\
I(F,\lambda,z) &= \begin{cases}
\displaystyle - \int_0^\lambda g(w- \lambda) \,e^{-zw}\, \dw,\, \text{ if }\re(z)\ge -1/4;\\[2.5ex]
\displaystyle \int_{-\infty}^0 g(w- \lambda) \,e^{-zw}\, \dw, \, \text{ if }\re(z) < -1/4.
\end{cases}
\end{align*}
The value $-1/4$ is not essential in the definitions above; any $\tau<0$ could be used in its place. Equations \eqref{A-rep} and \eqref{Sec2_A_equal_A_1} can be rewritten as
\begin{align}\label{A-rep2}
\A(F,\lambda,z) - H(z) e^{- \lambda z}= F(z) I(F,\lambda,z)
\end{align}
 for all $z \in \C$.  We establish in the next two lemmas some straightforward but crucial properties of these functions.

\begin{lemma}\label{bdd-lemma} The functions $z\mapsto I(F,\lambda,z)$ and $z\mapsto H(z) e^{-\lambda z}$ are bounded in the complex plane $($by a constant depending on $\lambda$$)$.
\end{lemma} 

\begin{proof} We consider the function $I$.  For $\re(z)\ge -1/4$ we have
$$|I(F,\lambda,z)| = \left|\int_0^{\lambda} g(w- \lambda)\, e^{-zw}\, \dw\right| = O_\lambda(1).$$
It follows from Lemma \ref{Sec2_lem1}  that $g=g_{0-}$ is identically zero for $t\ge 0$, and it follows from Lemma \ref{g-sign-changes} (iii) (with $\tau = -1/8$, say) that $|I(F,\lambda,z)|\le C|\re(z)+1/8|^{-1}$ for $\re (z) < -1/4$. The claim for $z\mapsto H(z) e^{-\lambda z}$ is evident.
\end{proof}

\begin{lemma} \label{Lem8_Sec2}
Let $\varepsilon = \pm 1$ be the sign of $F$ in the interval $(-\infty,0)$. For all $x \in \R$ we have
\begin{align}\label{A-ineq}
\varepsilon \,F(x) \Big\{e^{- \lambda x} - \A(F,\lambda,x) \Big\} \ge 0.
\end{align}
\end{lemma}

\begin{proof} From Lemma \ref{g-sign-changes} (i), the function $g$ is of one sign, equal to the sign of $F$ in $(-\infty,0)$ (in fact, a considerably stronger statement regarding the derivatives of $g$ is true, cf. \cite[Chapter IV, Theorems 5.1 and 5.3]{HW}, but we do not require it here). The inequality now follows directly from the definition of $\A$ in \eqref{A-rep}. 
\end{proof}

Note that, with our normalization \eqref{F-absconv}, we have $\varepsilon = 1$ if  $F(0)\neq 0$ and $\varepsilon =-1$ if $F(0) =0$. 

\subsection{Interpolating the Gaussian} Let $A:\C \to \C$ be an {\it even} Laguerre-P\'{o}lya function of exponential type $\tau(A)$. Assume that $A(0)=1$ and label the positive zeros of $A$ by
$$0<a_1\le a_2 \le \ldots $$
with repetition according to multiplicity. Note that for all $z \in \C$ the Hadamard product representation
$$A(z) = \prod_{j=1}^\infty \left(1-\frac{z^2}{a_j^2}\right)$$
is valid (the absence of exponential factors in the product is crucial for our purposes). The product converges uniformly and absolutely in compact sets of $\C$. Define $F_A:\C \to \C$ by
\begin{equation}\label{Sec2_Def_F_A}
F_A(z) = \prod_{j=1}^\infty \left( 1- \frac{z}{a_j^2}\right)
\end{equation}
and observe that $F_A$ is also a Laguerre-P\'{o}lya function, and that it verifies
$$A(z) = F_A\big(z^2\big)$$
for all $z \in \C$.

\smallskip

In a similar way, let $B:\C \to \C$ be an even Laguerre-P\'{o}lya function of exponential type $\tau(B)$, with a double zero at the origin, and positive zeros listed as 
$$0<b_1\le b_2 \le \ldots $$
If we assume that $B''(0) = 2$, we have
$$B(z) = z^2\prod_{j=1}^\infty \left(1-\frac{z^2}{b_j^2}\right),$$
and we may define $F_B:\C \to \C$ analogously to \eqref{Sec2_Def_F_A} in order to have $B(z) = F_B\big(z^2\big)$. Observe that $F_A$ and $F_B$ verify \eqref{zero-conv} and \eqref{F-absconv}; $F_A$ with $k=0$ and $F_B$ with $k=1$.

\begin{lemma}\label{L-thm} Define functions $L$ and $M$ by 
\begin{align*}
L(A, \lambda, z) &= \A(F_A, \pi\lambda, z^2);\\
M(B, \lambda, z) & = \A(F_B, \pi\lambda, z^2).
\end{align*}
The following propositions hold:
\begin{enumerate}
\item[(i)] The function $z \mapsto L(A, \lambda, z)$ is an entire function of exponential type at most $\tau(A)$, and for every $x \in \R$ the inequality
\begin{equation}\label{Sec2_Lem9_ineq_L}
A(x) \Big\{e^{-\pi \lambda x^2} - L(A,\lambda,x)\Big\}\ge 0
\end{equation}
holds. Moreover, for all $\xi$ with $A(\xi) =0$ we have
\begin{equation}\label{Sec2_Lem9_id_L}
L(A,\lambda,\xi) = e^{-\pi \lambda \xi^2}.
\end{equation}

\item[(ii)] The function $z \mapsto M(B, \lambda, z)$ is an entire function of exponential type at most $\tau(B)$, and for every $x \in \R$ the inequality
\begin{align}\label{Sec2_Lem9_ineq_M}
B(x) \Big\{ M(B,\lambda,x) - e^{-\pi \lambda x^2} \Big\} \ge 0
\end{align}
holds.  Moreover, for all $\xi$ with $B(\xi)=0$ we have 
\begin{align}\label{Sec2_Lem9_id_M}
M(B,\lambda,\xi) = e^{-\pi \lambda \xi^2}.
\end{align}

\end{enumerate}
\end{lemma}

\begin{proof} Since $z \mapsto \A(F_A, \pi\lambda, z)$ is entire, it follows that $z \mapsto L(A, \lambda, z)$ is entire. From \eqref{A-rep2} we have
$$L(A,\lambda,z) = H\big(z^2\big) \,e^{-\pi \lambda z^2} + F_A\big(z^2\big) \,I(F_A,\pi\lambda,z^2),$$
and Lemma \ref{bdd-lemma} implies that there exist constants $C, D>0$ such that
\[
|L(A,\lambda,z)| \le C + D \,|F_A\big(z^2\big)| = C + D \,|A(z)|
\]
for all $z \in \C$. It follows that $z \mapsto L(A, \lambda, z)$ has exponential type at most $\tau(A)$. The inequality \eqref{Sec2_Lem9_ineq_L} follows from  Lemma \ref{Lem8_Sec2}, while identity \eqref{Sec2_Lem9_id_L} follows from the definition of $\A$ in \eqref{A-rep}. This completes the proof of (i).  The proof of (ii) is analogous.
\end{proof}

\section{Proof of Theorem \ref{Thm1}}\label{Sec3}

The general strategy we follow here is similar to \cite[Section 3]{CL3}. Let $E$ be a Hermite-Biehler function satisfying properties (P1) - (P4), and let $A$ and $B$ be the companion functions defined by \eqref{Intro_def_A_B}. Property (P1) implies that $A$ and $B$ are Laguerre-P\'{o}lya functions of exponential type at most $\tau(E)$. Property (P3) implies that $A$ is even and $B$ is odd, while property (P2) implies that $A(0) \neq 0$ and $B$ has a simple zero at the origin. Further details are given in \cite[Lemmas 12 and 13]{CL3}.

\begin{lemma}\label{L1-element} 
Let $\lambda>0$. Let $z\mapsto L(A^2, \lambda, z)$ and $z\mapsto M(B^2,\lambda ,z)$ be defined as in Lemma \ref{L-thm}. Then
$$x\mapsto L(A^2,\lambda, x) - e^{-\pi\lambda x^2}$$
and
$$x\mapsto M(B^2, \lambda, x) - e^{-\pi\lambda x^2}$$
belong to $L^1(\R,|E(x)|^{-2} \,\dx)$.
\end{lemma}

\begin{proof} We note from \eqref{A-rep} that
$$\left|\frac{L(A^2,\lambda, x) - e^{-\pi\lambda x^2}}{A^2(x)}\right| = \left| \int_0^{\pi \lambda} \,g(w-\pi\lambda) \,e^{-x^2w} \dw \right| \le \frac{c}{x^2}\,,$$
since $g$ is bounded. This implies the first claim since $|A/E|$ is bounded by $1$ on the real line. The second claim is shown analogously.
\end{proof}

\begin{lemma}\label{UU*-decomp} 
Let $E$ be a Hermite-Biehler function of bounded type in $\U$ with exponential type $\tau(E)$. 

\begin{itemize}

\item[(i)] If $M: \C \to \C$ is a real entire function of exponential type at most $2\tau(E)$, which is nonnegative on $\R$ and belongs to $L^1(\R,|E(x)|^{-2}\, \dx)$, then there exists $W \in \H(E)$ such that 
\begin{equation}\label{Sec3_Lem10_1st_claim}
M(z) = W(z) W^*(z)
\end{equation}
for all $z\in\C$. 

\smallskip

\item[(ii)] In particular, if $\lambda>0$ and 
\begin{equation}\label{Sec3_cond_Gaussian_E}
\int_{-\infty}^\infty e^{-\pi\lambda x^2}\, |E(x)|^{-2}\, \dx <\infty,
\end{equation}
there exist $U,V\in\mc{H}(E)$ so that
\begin{align}
M(B^2,\lambda,z) & = U(z) U^*(z), \label{Sec3_eqM_U_U_star}\\
L(A^2,\lambda,z)&= U(z) U^*(z) - V(z) V^*(z) \label{Sec3_eqL_U_U_star}
\end{align}
for all $z\in\C$.
\end{itemize}
\end{lemma}

\begin{proof} Part (i) is \cite[Lemma 14]{CL3}. To prove part (ii), note by Lemma \ref{L-thm} that $z \mapsto L(A^2,\lambda,z)$ and $z \mapsto M(B^2,\lambda,z)$ have exponential type at most $2\tau(E)$. From Lemma \ref{L1-element} and \eqref{Sec3_cond_Gaussian_E}, it follows that $x \mapsto M(B^2, \lambda, x)$ is an element of $L^1(\R,|E(x)|^{-2} \dx)$ and this gives the representation \eqref{Sec3_eqM_U_U_star} as a particular case  of \eqref{Sec3_Lem10_1st_claim}. Representation \eqref{Sec3_eqL_U_U_star} follows similarly, considering the real entire function $z \mapsto M(B^2,\lambda,z) - L(A^2,\lambda,z)$, which is nonnegative on $\R$.
\end{proof}

\begin{proof}[Proof of Theorem \ref{Thm1}]  Let $M:\C \to \C$ be an entire function of exponential type at most $2\tau(E)$ such that $M(x) \ge e^{-\pi\lambda x^2}$ for all $x \in \R$ and
$$\int_{-\infty}^\infty M(x) \,|E(x)|^{-2} \,\dx <\infty. $$
Lemma \ref{UU*-decomp} implies that $M = WW^*$ for some $W\in \mc{H}(E)$. Since $B\notin \mc{H}(E)$ by hypothesis (P4), we are in position to use \cite[Theorem 22]{B} to conclude that $\{z \mapsto K(\xi, z);\ B(\xi) = 0\}$ is a complete orthogonal set in $\H(E)$. By Plancherel's identity we have
\begin{align}\label{M-extremality}
\begin{split}
\int_{-\infty}^\infty M(x) \,|E(x)|^{-2}\, \dx &= \int_{-\infty}^\infty \left| \frac{W(x)}{E(x)}\right|^2 \dx = \sum_{B(\xi) =0} \frac{|W(\xi)|^2}{K(\xi,\xi)} \ge  \sum_{B(\xi) =0} \frac{e^{-\pi\lambda \xi^2}}{K(\xi,\xi)},
\end{split}
\end{align}
which proves \eqref{Intro-M-integral}. From Lemma \ref{L-thm} we know that $z\mapsto M(B^2,\lambda,z)$ is an entire function of exponential type at most $2\tau(E)$ that satisfies
$$M(B^2,\lambda,x) \geq e^{-\pi\lambda x^2}$$
for all $x \in \R$. Moreover, it follows from  \eqref{Sec2_Lem9_id_M} and \eqref{Sec3_eqM_U_U_star} that we have equality in \eqref{M-extremality} for $M(z) = M(B^2,\lambda,z)$. 

\smallskip

Now let $L:\C \to \C$ be an entire function of exponential type at most $2\tau(E)$ such that $L(x) \le e^{-\pi\lambda x^2}$ for all $x \in \R$ and
$$\int_{-\infty}^\infty L(x) \,|E(x)|^{-2} \,\dx <\infty.$$
We use the fact that {\it there exists a majorant} $z \mapsto M(B^2,\lambda,z)$ in $L^1(\R,|E(x)|^{-2} \dx)$. This allows us to apply Lemma \ref{UU*-decomp} to the function $z \mapsto M(B^2,\lambda,z) - L(z)$ (which is nonnegative on the real axis, belongs to $L^1(\R,|E(x)|^{-2} \dx)$ and has exponential type at most $2\tau(E)$). Together with \eqref{Sec3_eqM_U_U_star}, this shows that  $L = UU^* - TT^*$ with $U,T \in \H(E)$. Since $A \notin \H(E)$ by hypothesis (P4), a new application of \cite[Theorem 22]{B} shows that $\{z \mapsto K(\xi, z);\ A(\xi) = 0\}$ is also a complete orthogonal set in $\H(E)$. Plancherel's identity then gives us
\begin{align}\label{L-extremality}
\begin{split}
\int_{-\infty}^\infty L(x) \,|E(x)|^{-2}\, \dx &= \int_{-\infty}^\infty  \frac{|U(x)|^2 - |T(x)|^2}{|E(x)|^2}\, \dx = \sum_{A(\xi) =0} \frac{|U(\xi)|^2 - |T(\xi)|^2}{K(\xi,\xi)} \le  \sum_{A(\xi) =0} \frac{e^{-\pi\lambda \xi^2}}{K(\xi,\xi)},
\end{split}
\end{align}
which proves \eqref{Intro-L-integral}. From Lemma \ref{L-thm} we know that $z \mapsto L(A^2,\lambda,z)$ is an entire function of exponential type at most $2\tau(E)$ that satisfies
$$L(A^2,\lambda,x) \leq e^{-\pi\lambda x^2}$$
for all $x \in \R$, and it follows from \eqref{Sec2_Lem9_id_L} and \eqref{Sec3_eqL_U_U_star} that we have equality in \eqref{L-extremality} for $L(z) = L(A^2,\lambda,z)$.
\end{proof}

\section{Proof of Theorem \ref{Thm2}}

\subsection{Preliminaries} 

\subsubsection{Homogeneous spaces} We start by recalling the basic properties of the spaces $\H(E_{\nu})$, which are examples of homogeneous de Branges spaces \cite[Section 50]{B}. Further details are given in \cite[Section 5]{HV} and \cite[Section 4]{CL3}.

\smallskip

Recall that $\nu >-1$ is a given parameter and 
$$E_{\nu}(z) := A_{\nu}(z) - iB_{\nu}(z)\,,$$
with $A_{\nu}$ and $B_{\nu}$ given by \eqref{Intro_A_nu} and \eqref{Intro_B_nu}, respectively. From \cite[Problems 227 and 228]{B} we have that $E_{\nu}$ is in fact a Hermite-Biehler function with no zeros on the real line. Moreover, $E_{\nu}$ has bounded type in $\U$ with $v(E_{\nu}) = \tau(E_{\nu}) = 1$. By definition, $A_{\nu}$ is even while $B_{\nu}$ is odd. This plainly implies that $z \mapsto E_{\nu}(iz)$ is real entire. This accounts for properties (P1) - (P3) for $E_{\nu}$. The next lemma collects the relevant properties of the spaces $\H(E_{\nu})$ for our purposes.

\begin{lemma}\label{Sec5_rel_facts}
Let $\nu > -1$. The following properties hold:
\begin{enumerate}
\item[(i)] There exist positive constants $a_\nu,b_\nu$ such that 
\begin{align}\label{Lem17_i}
a_\nu |x|^{2\nu+1} \le |E_{\nu}(x)|^{-2} \le b_\nu |x|^{2\nu+1}
\end{align}
for all $x \in \R$ with $|x|\geq1$.
\smallskip

\item[(ii)] For $F\in\H(E_\nu)$ we have the identity 
\begin{align} \label{Lem17_ii}
\int_{-\infty}^\infty |F(x)|^{2}\,|E_{\nu}(x)|^{-2}\, \dx = c_\nu \int_{-\infty}^\infty |F(x)|^2 \,|x|^{2\nu+1} \,\dx\,,
\end{align}
with $c_\nu = \pi \,2^{-2\nu-1}\, \Gamma(\nu+1)^{-2}$.

\smallskip

\item[(iii)] An entire function $F$ belongs to $\H(E_\nu)$ if and only if $F$ has exponential type at most $1$ and
\begin{equation}\label{Sec4_eq0_int_cond}
\int_{-\infty}^\infty |F(x)|^2 \,|x|^{2\nu+1}\, \dx <\infty.
\end{equation}
\end{enumerate}
\end{lemma}

\begin{proof} We refer the reader to \cite[Section 5]{HV}.
\end{proof}

In particular, the identity \eqref{Lem17_ii} is fundamental for our purposes, and provides a good example of the power (and a bit of mystery) of the de Branges space machinery. From \eqref{Intro_Bessel1}, \eqref{Intro_Bessel2} and \eqref{Asymptotic_Bessel_functions} we see that $A_{\nu}$ and $B_{\nu}$ do not verify \eqref{Sec4_eq0_int_cond} and thus $A_{\nu}, B_{\nu} \notin \H(E_{\nu})$. This establishes the remaining property (P4).

\subsubsection{Symmetrization techniques} Before we move on to the proof of Theorem \ref{Thm2} we also to need to recall the basic symmetrization mechanisms to deal with radial functions in several variables. This is contained in \cite[Section 6]{HV} and we again replicate here a brief collection of the main facts for the reader's convenience.

\smallskip

We start with our extension operator. If $F:\C\to \C$ is an even entire function with power series representation
\begin{align*}
F(z) = \sum_{k=0}^\infty c_k z^{2k}\,,
\end{align*}
we define the extension $\psi_N(F):\C^N \to \C$ by
\begin{align*}
\psi_N(F)({\bf z}) = \sum_{k=0}^\infty c_k (z_1^2 + \ldots + z_N^2)^k.
\end{align*}
We also consider a suitable symmetrization operator. For an entire function $\F: \C^N \to \C$, with $N >1$, we define its radial symmetrization $\widetilde{\F}:\C^N \to \C$ by 
\begin{equation*}
\widetilde{\F}(\z) = \int_{SO(N)}\F(R\z)\,\d \sigma(R),
\end{equation*}
where $SO(N)$ denotes the compact topological group of real orthogonal $N \times N$ matrices with determinant $1$, with associated Haar measure $\sigma$ normalized so that $\sigma(SO(N)) = 1$. If $N=1$ we set $\widetilde{\F}(z) = \tfrac12\{\F(z) + \F(-z)\}$.

\begin{lemma}\label{type-to-ntype} 
The following propositions hold:
\smallskip
\begin{itemize}
\item[(i)] Let $F: \C \to \C$ be an even entire function. Then $F$ has exponential type if and only if $\psi_N(F)$ has exponential type, and $\tau(F) = \tau(\psi_N(F))$. 

\smallskip

\item[(ii)] Let $\F: \C^N \to \C$ be an entire function. Then $\widetilde{\F}:\C^N \to \C$ is an entire function with power series expansion of the form
\begin{equation}\label{Sec4_PWRep}
\widetilde{\F}({\bf z}) = \sum_{k=0}^\infty c_k (z_1^2 + \ldots + z_N^2)^k.
\end{equation}
Moreover, if $\F$ has exponential type then $\widetilde{\F}$ has exponential type and $\tau\big(\widetilde{\F}\big) \leq \tau(\F)$.

\end{itemize}
\end{lemma}

\begin{proof} These are \cite[Lemmas 18 and 19]{HV}.
\end{proof}

\subsection{Proof of Theorem \ref{Thm2}} Recall that we write $\F_{\lambda}(\x) = e^{-\pi \lambda |\x|^2}$. 

\subsubsection{Proof of part {\rm (i)}} For $\kappa >0$ observe that $\L \in \E_{\delta}^{N-} (\mc{F}_{\lambda})$ if and only if $\x \mapsto \L(\kappa \x) \in \E_{\kappa\delta}^{N-} (\mc{F}_{\kappa^2\lambda})$. An analogous property holds for the majorants. This is enough to conclude part (i).

\subsubsection{Proof of part {\rm (ii)}} Let $L \in \E_{\delta}^{1-} (\mc{F}_{\lambda})$ be such that 
$$\int_{-\infty}^\infty \left\{ e^{-\pi\lambda x^2} - L(x)\right\} \,|x|^{2\nu+1} \,\dx < \infty.$$
By replacing $L$ by $\widetilde{L}$, we may assume that $L$ is even, and hence has a power series expansion
\[
L(z) = \sum_{k=0}^\infty c_{k} z^{2k}.
\]
By Lemma \ref{type-to-ntype} (i) we have that $\psi_{N}(L) \in \E_{\delta}^{N-} (\mc{F}_{\lambda})$, and a change to radial variables gives us
\begin{align*}
\int_{\R^N} \Big\{e^{-\pi \lambda |\x|^2}  - \psi_N(L)(\x)\Big\}\, |\x|^{2\nu+2 - N}\,\dxx  = \tfrac12\, \omega_{N-1} \int_{-\infty}^\infty \Big\{ e^{-\pi \lambda x^2} - L(x)\Big\} \,|x|^{2\nu+1} \,\dx.
\end{align*}
Hence
\begin{equation*}\label{Sec5_eq3_pI}
U_{\nu}^{N-}(\delta, \lambda) \leq  \tfrac12\, \omega_{N-1} \,U_{\nu}^{1-}(\delta, \lambda).
\end{equation*}

\smallskip

On the other hand, let $\L \in \E_{\delta}^{N-} (\mc{F}_{\lambda})$ be such that 
\begin{align*}
\int_{\R^N} \Big\{e^{-\pi \lambda|\x|^2} - \L(\x)\Big\}\, |\x|^{2\nu+2 - N}\,\dxx < \infty.
\end{align*}
By Lemma \ref{type-to-ntype} (ii) we have that $\widetilde{\L} \in \E_{\delta}^{N-} (\mc{F}_{\lambda})$ and that it has a power series expansion 
\begin{equation*}
\widetilde{\L}({\bf z}) = \sum_{k=0}^\infty c_k (z_1^2 + \ldots + z_n^2)^k.
\end{equation*}
Define the entire function $L:\C \to \C$ by
\begin{equation*}
L(z) = \widetilde{\L}(z,0,0, \ldots, 0) = \sum_{k=0}^\infty c_k z^{2k}.
\end{equation*}
Another application of Lemma \ref{type-to-ntype} (i) gives us that $L \in \E_{\delta}^{1-} (\F_{\lambda})$ and we arrive at

\begin{align*}
\tfrac12\, \omega_{N-1} \int_{-\infty}^\infty  \Big\{ e^{-\pi \lambda x^2}  - L(x)\Big\} \,|x|^{2\nu+1} \,\dx
 &=  \int_{\R^N} \Big\{e^{-\pi \lambda|\x|^2} - \widetilde{\L}(\x)\Big\}\, |\x|^{2\nu+2 - N}\,\dxx\\
 & =  \int_{\R^N} \Big\{e^{-\pi \lambda|\x|^2} - \L(\x)\Big\}\, |\x|^{2\nu+2 - N}\,\dxx.
\end{align*}
Hence
\begin{equation*}\label{Sec5_eq3_pII}
 \tfrac12\, \omega_{N-1} \,U_{\nu}^{1-}(\delta, \lambda) \leq U_{\nu}^{N-}(\delta, \lambda),
\end{equation*}
and this concludes the proof for the minorant case. The majorant case is treated analogously.

\subsubsection{Proof of part {\rm (iii)}} We are now interested in computing $U_{\nu}^{1\pm}(2, \lambda)$. Let us start with the majorant case. First observe that 
\begin{equation*}
\int_{-\infty}^{\infty} e^{-\pi \lambda x^2}\,|x|^{2\nu +1} \,\dx = \frac{\Gamma(\nu +1)}{(\pi \lambda)^{\nu +1}}.
\end{equation*}
Let $M:\C \to \C$ be an entire function of exponential type at most $2$ such that $M(x) \geq e^{-\pi \lambda x^2}$ for all $x \in \R$ and 
\begin{equation*}
\int_{-\infty}^{\infty}M(x)\,|x|^{2\nu +1}\,\dx <\infty.
\end{equation*}
From \eqref{Lem17_i} we obtain
\begin{equation*}
\int_{-\infty}^{\infty}M(x)\,|E_{\nu}(x)|^{-2}\,\dx <\infty.
\end{equation*}
It follows from Lemma \ref{UU*-decomp} that $M = WW^*$ for some $W \in \mc{H}(E_{\nu})$. From Theorem \ref{Thm1} and the key identity \eqref{Lem17_ii} we have
\begin{align}\label{Sec5_eq3_majorant_B}
\begin{split}
\sum_{B_{\nu}(\xi)=0} \frac{e^{-\pi \lambda\xi^2}}{K_{\nu}(\xi,\xi)} & \leq \int_{-\infty}^{\infty} M(x) \,|E_{\nu}(x)|^{-2}\,\dx\\
& =  \int_{-\infty}^{\infty} |W(x)|^2 \,|E_{\nu}(x)|^{-2}\,\dx\\
& =  c_{\nu}\int_{-\infty}^{\infty} |W(x)|^2 \,|x|^{2\nu +1}\,\dx\\
& =  c_{\nu}\int_{-\infty}^{\infty}M(x)\,|x|^{2\nu +1}\,\dx.
\end{split}
\end{align}
Moreover, we have seen that the entire function $z \mapsto M(B_{\nu}^2, \lambda, z)$ of exponential type at most $2$ realizes the equality in \eqref{Sec5_eq3_majorant_B}. Therefore
\begin{equation*}
U_{\nu}^{1+}(2, \lambda) = \sum_{B_{\nu}(\xi)=0} \frac{e^{-\pi \lambda\xi^2}}{c_{\nu}K_{\nu}(\xi,\xi)} - \frac{\Gamma(\nu +1)}{(\pi \lambda)^{\nu +1}}.
\end{equation*} 
The minorant case follows analogously, using the fact that {\it there exists a majorant} $z \mapsto M(B_{\nu}^2, \lambda, z)$. This allows us to write a general minorant $L$ as $UU^* - TT^*$, as done in the proof of Theorem \ref{Thm1}.

\subsubsection{Proof of part {\rm (iv)}} For $N \geq 1$ and $\delta = 2$ we define $\L_{\nu,N}(2, \lambda, \z) := \psi_N \big(L(A_{\nu}^2, \lambda, \cdot)\big)(\z)$  and $\M_{\nu,N}(2, \lambda, \z) := \psi_N \big(M(B_{\nu}^2, \lambda, \cdot)\big)(\z)$. These functions have exponential type at most $2$ and satisfy
$$\L_{\nu,N}(2, \lambda, \x)  \leq e^{-\pi \lambda|\x|^2} \leq \M_{\nu,N}(2, \lambda, \x)$$
for all $\x \in \R^N$. Moreover,
\begin{align*}
\int_{\R^N} \Big\{ \M_{\nu,N}(2, \lambda, \x) - e^{-\pi \lambda|\x|^2}\Big\}\,|\x|^{2\nu +2 - N}\,\dxx & = \tfrac12\, \omega_{N-1}\int_{-\infty}^{\infty} \Big\{ M(B_{\nu}^2,\lambda, x) - e^{-\pi \lambda x^2}\Big\}\,|x|^{2\nu +1}\,\dx\\[0.5em]
& =  \tfrac12\, \omega_{N-1} \,U_{\nu}^{1+}(2,\lambda) \\[0.5em]
& = U_{\nu}^{N+}(2,\lambda).
\end{align*}
The computation for the minorant is analogous. The existence of extremal functions for general $\delta >0$ follows by the change of variables given in part (i). This completes the proof.

\section{Proof of Theorem \ref{Thm3} - Gaussian subordination}\label{Sec_Gaussian_Sub}

\subsection{Preliminaries} We start by showing that the asymptotic estimate \eqref{Intro_open_asymp} holds. We have established the integral representation
\[
U_\nu^{1+} (2,\lambda) = -\int_{-\infty}^\infty  \left\{B_\nu^2(x) \int_0^{\pi \lambda} g(w-\pi \lambda) \,e^{-x^2 w}\, \dw\right\} \,|x|^{2\nu +1}\,\dx\,,
\]
where $g = g_{0-}$ for the Laguerre-P\'{o}lya function $F_{B_\nu^2}$ as in \eqref{gc-def} and \eqref{Sec2_Def_F_A}. We note that $B_\nu^2/|E_{\nu}|^2$ is bounded by $1$ on the real line, and therefore, using \eqref{Lem17_i}, we see that $B_\nu^2(x)\,|x|^{2\nu +1}$ is also bounded on the real line. Moreover, $g(t) =0$ for $t\ge 0$ and $g\in C^\infty(\R)$. This implies, in particular, that $g^{(k)}(0) =0$ for all $k\in\N$, which leads to $|g(w-\pi \lambda)|\le c_k(\pi \lambda-w)^k$ for $\lambda\le 1$, say, and all $0\le w\le \pi \lambda$. An application of Fubini's theorem gives \eqref{Intro_open_asymp} for the majorant. The minorant case follows analogously using $A_\nu$.

\smallskip

If the nonnegative Borel measure $\mu$ on $(0,\infty)$ satisfies \eqref{Intro_mu_1} we define
\begin{equation}\label{Sec6_eq1_int1}
D_{\nu,N}^-(\delta, \mu,{\bf x}) = \int_0^{\infty} \Big\{e^{-\pi \lambda |{\bf x}|^2} - \mc{L}_{\nu,N}(\delta,\lambda, \x)\Big\} \, \dmu,
\end{equation}
and if $\mu$ satisfies \eqref{Intro_mu_2}  we define
\begin{equation}\label{Sec6_eq1_int2}
D_{\nu,N}^+(\delta, \mu,{\bf x}) = \int_0^{\infty} \Big\{\mc{M}_{\nu,N}(\delta, \lambda, \x) - e^{-\pi \lambda |{\bf x}|^2}\Big\} \, \dmu,
\end{equation}
where $\z \mapsto  \mc{L}_{\nu,N}(\delta, \lambda, \z)$ and $\z \mapsto  \mc{M}_{\nu,N}(\delta, \lambda, \z)$ are the extremal functions of exponential type at most $\delta$ defined in Theorem \ref{Thm2} (iv). Integration in $\x$ with respect to $|\x|^{2\nu + 2 - N}\d\x$ and an application of Tonelli's theorem implies that the  nonnegative functions $\x \mapsto D_{\nu,N}^{\pm}(\delta, \mu,{\bf x})$ are radial and belong to $L^1(\R^N, |\x|^{2\nu +2 - N}\,\dxx)$, hence the integrals \eqref{Sec6_eq1_int1} and \eqref{Sec6_eq1_int2} converge almost everywhere.

\begin{lemma}\label{Sec6_Lem23}
The following properties hold:
\smallskip
\begin{enumerate}
\item[(i)] If $\mu$ satisfies \eqref{Intro_mu_1}, then the nonnegative function $\x \mapsto D_{\nu,N}^{-}(\delta, \mu,{\bf x})$ is continuous for $\x \neq 0$. Moreover, $D_{\nu}^{-}(\delta, \mu,{\bf x}) =0 $ whenever $A_{\nu}(\delta |\x|/2)=0$.
\smallskip
\item[(ii)] If $\mu$ satisfies \eqref{Intro_mu_2}, then the nonnegative function $\x \mapsto D_{\nu,N}^{+}(\delta, \mu,{\bf x})$ is continuous. Moreover, $D_{\nu}^{+}(\delta, \mu,{\bf x}) =0 $ whenever $B_{\nu}(\delta|\x|/2)=0$.
\end{enumerate}
\end{lemma}

\begin{proof} We restrict ourselves to the case $\delta =2$. The general case follows by a standard dilation argument.

\smallskip

\noindent{\it Part} (i). Recall that $A_{\nu}(0)\neq  0$. Let $g = g_{0-}$ for the Laguerre-P\'{o}lya function $F_{A_\nu^2}$ as in \eqref{gc-def} and \eqref{Sec2_Def_F_A}, and note by Lemma \ref{g-sign-changes} that $g\ge 0$ on $\R$. We have
\begin{align*}
e^{-\pi \lambda x^2} - L(A_{\nu}^2, \lambda,x) =  A_{\nu}^2(x) \int_{0}^{\pi \lambda}  g(w - \pi\lambda) \,e^{-x^2w} \,\dw.
\end{align*}
Since $g\ge 0$ on $\R$ it follows that $H_{\lambda}^{-}$ defined by
\begin{equation*}
H_{\lambda}^{-}(x) =\int_{0}^{\pi\lambda}   g(w-\pi \lambda) \,e^{-x^2w} \,\dw
\end{equation*}
is nondecreasing on $(-\infty,0]$ for every $\lambda>0$. For $x\leq0$ we define
\begin{equation*}
d_\nu^-(2, \mu,x) := A_{\nu}^2(x) \int_0^\infty H_{\lambda}^{-}(x)\, \dmu,
\end{equation*}
and note that this function is continuous on $(-\infty, 0)$ by dominated convergence (we may have \linebreak $\limsup_{x \to 0^-} d_\nu^-(2, \mu,x) = \infty$). It follows that $d_\nu^-(2, \mu,\xi)=0$ for every $\xi<0$ with $A_{\nu}(\xi)=0$. Since $D_{\nu,N}^-(2, \mu,\x) = d_\nu^-(2, \mu,-|\x|)$, our claim follows. 

\smallskip

\noindent {\it Part} (ii). Similarly, consider $g = g_{0-}$ for the Laguerre-P\'{o}lya function $F_{B_\nu^2}$ as in \eqref{gc-def} and \eqref{Sec2_Def_F_A}. Now we have $g \leq 0$ on $\R$ and 
\begin{align*}
M(A_{\nu}^2, \lambda,x) - e^{-\pi \lambda x^2} =  - B_{\nu}^2(x) \int_{0}^{\pi \lambda}  g(w - \pi\lambda) \,e^{-x^2w} \,\dw.
\end{align*}
Define 
\begin{equation*}
H_{\lambda}^{+}(x) = - \int_{0}^{\pi\lambda}   g(w-\pi \lambda) \,e^{-x^2w} \,\dw
\end{equation*}
and note that this is a nonnegative function which is nondecreasing on $(-\infty,0]$ for every $\lambda>0$. For $x\leq0$ define
\begin{equation*}
d_\nu^+(2,\mu,x) := B_{\nu}^2(x) \int_0^\infty H_{\lambda}^{+}(x)\, \dmu,
\end{equation*}
and again, we note that this function is continuous on $(-\infty, 0)$ by dominated convergence. Thus $d_\nu^+(2, \mu,\xi)=0$ for every $\xi<0$ with $B_{\nu}(\xi)=0$. 

\smallskip

To analyze the situation as $x \to 0^-$, we use integration by parts to get
\begin{equation*}
H_{\lambda}^{+}(x) = - \frac{g(-\pi \lambda)}{x^2} - \frac{1}{x^2}  \int_{0}^{\pi\lambda}   g'(w-\pi \lambda) \,e^{-x^2w} \,\dw,
\end{equation*}
and thus
\begin{equation}\label{Sec5_esto_dg}
d_\nu^+(2, \mu,x) = \frac{B_{\nu}^2(x)}{x^2} \int_0^{\infty}\left(- g( - \pi \lambda) - \int_{0}^{\pi\lambda}   g'(w-\pi \lambda) \, e^{-x^2w}\,\dw \right) \, \dmu.
\end{equation}
Observe now that
\begin{equation}\label{Sec5_est1_dg}
\int_0^{\infty} |g(-\pi \lambda)|\, \dmu < \infty
\end{equation}
and
\begin{equation}\label{Sec5_est2_dg}
\int_0^{\infty} \int_{0}^{\pi\lambda}   \big|g'(w-\pi \lambda)\big| \,\dw \,\dmu < \infty.
\end{equation}
In fact, \eqref{Sec5_est1_dg} follows from the fact that $g$ is bounded and $g^{(k)}(0) = 0$ for any $k \geq0$. To arrive at \eqref{Sec5_est2_dg}, we use the additional fact that $|g'(t)| \ll e^{\tau t}$ for some $\tau>0$ (Lemma \ref{g-sign-changes} (iii)), which holds since $g'$ is the frequency function associated to the Laguerre-P\'{o}lya function $F_{B_{\nu}^2}(z)/z$ on a vertical strip containing the origin. We remark that $F_{B_{\nu}^2}(z)/z$ does not have a zero at the origin anymore. We may then apply dominated convergence in \eqref{Sec5_esto_dg} to conclude that
\begin{equation*}
\lim_{x \to 0^-} d_\nu^+(2, \mu,x) = \left\{\lim_{x \to 0^-} \frac{B_{\nu}^2(x)}{x^2} \right\}\int_0^{\infty}\left(- g( - \pi \lambda) - \int_{0}^{\pi\lambda}   g'(w-\pi \lambda) \,\dw \right) \, \dmu = 0.
\end{equation*}
The result now follows since $D_{\nu,N}^+(2, \mu,\x) = d_\nu^+(2, \mu,-|\x|)$.
\end{proof}

The following lemma is crucial to prove the optimality part of Theorem \ref{Thm3}.

\begin{lemma}\label{Sec4_Cor22}
Let $\nu>-1$ and $G:\C \to \C$ be a real entire function of exponential type at most $2$ such that $G\in L^1(\R, |E_\nu(x)|^{-2}\, \dx)$. Then there exist $W,T\in\H(E_{\nu})$ such that 
\begin{align*}
G(z) = W(z)W^*(z) - T(z)T^*(z)
\end{align*}
for all $z \in \C$.
\end{lemma}

\begin{proof}
This is \cite[Theorem 23]{CL3}.
\end{proof}

\subsection{Proof of Theorem \ref{Thm3}} 

\subsubsection{Proof of parts {\rm (ii) and (iii)}} We shall prove these two parts together. Throughout this proof we denote by $\overline{B}(r)$ the closed ball of radius $r$ centered at the origin, and the parameter $\nu>-1$ is regarded as fixed.

\smallskip

\noindent {\it Existence}. 

\smallskip

We prove first the existence of minorants and majorants that interpolate our given function at the "correct'' radii. This is done in three steps. First we consider the particular dimension $N_{\nu} = \lceil 2\nu +2 \rceil$. A symmetrization argument will prove the existence for dimension $N=1$ and, finally, the radial extension will allow us to conclude the result for arbitrary $N$. The reason to start with the dimension $N_\nu$ lies in the fact that, in this dimension, all considered functions are known to define tempered distributions, which allows us to employ the Paley-Wiener theorem for distributions to establish the exponential type of the extremal functions. Since symmetrization and radial extension preserve the exponential type, their application allows us to establish the required exponential type first for $N=1$, and finally for arbitrary $N$.

\smallskip

We consider below the minorant problem (the majorant is analogous).

\smallskip

\noindent{{\it Step 1}. The case $N_{\nu} = \lceil 2\nu +2 \rceil$.}

\smallskip

Recall that $\G_{\mu, N_{\nu}}$ is assumed to be a tempered distribution satisfying
\begin{align}\label{Gmu-transform}
\widehat{\G}_{\mu, N_{\nu}}(\t) = \int_0^\infty \lambda^{-N_{\nu}/2} \, e^{-\pi|\t|^2/\lambda} \,\d\mu(\lambda)
\end{align}
for $|\t| >  \delta/2\pi$ (this follows from \eqref{Intro_Def_via_FT}).  The function $\x \mapsto D_{\nu, N_{\nu}}^-(\delta, \mu,{\bf x})$ defined in \eqref{Sec6_eq1_int2} belongs to $L^1(\R^{N_{\nu}}, |\x|^{2\nu + 2- N_{\nu}} \d\x)$. In particular, this function is integrable in $\overline{B}(1)$ (here it is important to have $N_{\nu} \geq 2\nu +2$) and thus defines a tempered distribution.

  \begin{lemma}\label{lemma_step_1} Let $\varphi$ be a Schwartz function that vanishes in $\overline{B}(\tfrac{\delta}{2\pi} + \varepsilon)$, where $\varepsilon >0$. Then
   \begin{equation}\label{Sec5_dist_eq_for_L}
 \int_{\R^{N_{\nu}}} \mc{L}_{\nu, N_{\nu}}(\delta,\lambda, \x) \,\widehat{\varphi}(\x) \,\d\x = 0.
\end{equation}
\end{lemma}

We give the proof of Lemma \ref{lemma_step_1} below. Assuming its validity, we show next that  
\begin{equation}\label{Sec5_FT_D_nu}
\widehat{D_{\nu, N_{\nu}}^-}(\delta, \mu,{\bf t}) = \int_0^\infty \lambda^{-N_{\nu}/2} \,e^{-\pi|\t|^2/\lambda} \,\d\mu(\lambda)
\end{equation}
for $|\t| > \delta/2\pi$. In fact, if $\varphi$ is a Schwartz function that vanishes in $\overline{B}(\tfrac{\delta}{2\pi} + \varepsilon)$, where $\varepsilon >0$, then
\begin{align}\label{Sec5_Dist_Arg}
\begin{split}
\int_{\R^{N_{\nu}}} D_{\nu, N_{\nu}}^-(\delta, \mu,{\bf x})\,\widehat{\varphi}(\x)\,\d\x & =\int_{\R^{N_{\nu}}} \int_0^{\infty} \Big\{e^{-\pi \lambda |{\bf x}|^2} - \mc{L}_{\nu, N_{\nu}}(\delta,\lambda, \x)\Big\} \,\widehat{\varphi}(\x) \, \dmu\,\d\x\\
& = \int_0^{\infty} \int_{\R^{N_{\nu}}} \Big\{e^{-\pi \lambda |{\bf x}|^2} - \mc{L}_{\nu, N_{\nu}}(\delta,\lambda, \x)\Big\} \,\widehat{\varphi}(\x) \,\d\x
\, \dmu\\
& = \int_0^{\infty}  \int_{\R^{N_{\nu}}} \lambda^{-N_\nu/2} e^{-\pi|\t|^2/\lambda}\,\varphi(\t)\,\d\t\, \dmu\\
& =  \int_{\R^{N_{\nu}}} \left( \int_0^{\infty}  \lambda^{-N_\nu/2} \, e^{-\pi|\t|^2/\lambda}\, \dmu\right)  \varphi(\t)\,\d\t.
\end{split}
\end{align}

\smallskip

From \eqref{Gmu-transform} and \eqref{Sec5_FT_D_nu} we see that the function
\begin{equation*}
\x \mapsto \G_{\mu, N_{\nu}}(\x) - D_{\nu, N_{\nu}}^-(\delta, \mu,{\bf x})
\end{equation*}
defines a tempered distribution, with Fourier transform supported on $\overline{B}(\delta/2\pi)$. It follows by the converse of the Paley-Wiener theorem for distributions that this function is equal almost everywhere on $\R^{N_{\nu}}$ to the restriction of an entire function of exponential type $\delta$. We call this entire function $\z \mapsto \L_{\nu, N_{\nu}}(\delta,\mu, \z)$. By Lemma \ref{Sec6_Lem23} and the assumption that $\x \mapsto \G_{\mu, N_{\nu}}(\x)$ is continuous on $\R^{N_{\nu}} \setminus \{0\}$, we find that
\begin{equation}\label{Sec5_eq_def_L_nu_rad}
\L_{\nu, N_{\nu}}(\delta,\mu, \x) = \G_{\mu, N_{\nu}}(\x) - D_{\nu, N_{\nu}}^-(\delta, \mu,{\bf x})
\end{equation}
for all $\x \neq 0$. Since $\x \mapsto D_{\nu, N_{\nu}}^-(\delta, \mu,{\bf x})$ is nonnegative, this implies that  
\begin{equation}\label{Sec5_cond1_import}
\L_{\nu, N_{\nu}}(\delta,\mu, \x) \leq \G_{\mu, N_{\nu}}(\x)
\end{equation}
for all $\x \in \R^{N_{\nu}}$, and by Lemma \ref{Sec6_Lem23} it follows that
\begin{equation}\label{Sec5_cond2_import}
\L_{\nu, N_{\nu}}(\delta,\mu, \x) = \G_{\mu, N_{\nu}}(\x)
\end{equation}
whenever $A_{\nu}(\delta |\x|/2)=0$. An application of Fubini's theorem gives us
\begin{equation}\label{Sec5_eva_case_N_nu}
\int_{\R^{N_{\nu}}} \big\{\G_{\mu, N_{\nu}}(\x) - \L_{\nu, N_{\nu}}(\delta,\mu,\x)\big\} \,|\x|^{2\nu+2-N_{\nu}} \,\d\x = \int_0^\infty U_\nu^{N_{\nu}-}(\delta,\lambda)\, \d\mu(\lambda).
\end{equation}
Hence, to conclude the proof of Theorem \ref{Thm3} for $N = N_\nu$ it remains to show Lemma \ref{lemma_step_1}.

\begin{proof}[Proof of Lemma \ref{lemma_step_1}] Recall from Theorem \ref{Thm2} that  
$$\mc{L}_{\nu, N_{\nu}}(\delta,\lambda, \z) = \psi_{N_{\nu}}(L(A_{\nu}^2,\lambda, \cdot))(\delta \z/2),$$ 
where $z \mapsto L(A_{\nu}^2,\lambda, z)$ is even, of exponential type at most $2$, and belongs to $L^1(\R, |x|^{2 \nu +1}\,\dx)$. This implies that $z \mapsto L(A_{\nu}^2,\lambda, z)$ must have at least one zero, for otherwise we would have $L(A_{\nu}^2,\lambda, z) = e^{cz +d}$ with $|c|\leq 2$, contradicting the fact that it belongs to $L^1(\R, |x|^{2 \nu +1}\,\dx)$. Then, if $a$ is a zero of $z \mapsto L(A_{\nu}^2,\lambda, z)$, so is $-a$ since the function is even. Hence, $J$ defined by
$$J(z) = \frac{L(A_{\nu}^2,\lambda, z)}{(z^2 - a^2)}$$
is entire and has exponential type at most $2$. It follows that 
$$\mc{J}(z) := \psi_{N_{\nu}}(J)(\delta \z/2) = \frac{\mc{L}_{\nu, N_{\nu}}(\delta,\lambda, \z)}{(\delta/2)^2(z_1^2 + \ldots + z_{N_{\nu}}^2) - a^2}$$
is entire of exponential type at most $\delta$. Since $\x \mapsto \mc{L}_{\nu, N_{\nu}}(\delta,\lambda, \x)$ belongs to $L^1(\R^{N_\nu}, |\x|^{2 \nu +2 - N_{\nu}}\,\d\x)$, we have that $\x \mapsto \mc{J}(\x) \in L^1(\R^{N_{\nu}})$ (here it is important to have $N_{\nu} \leq 2\nu +4$). By a classical result of Plancherel and P\'{o}lya \cite{PP}, $\mc{J}$ is bounded on $\R^{N_{\nu}}$ and thus belongs to $L^2(\R^{N_{\nu}})$. We now use the Paley-Wiener theorem in several variables \cite[Chapter III, Theorem 4.9]{SW} to conclude that $\widehat{\J}$ is a continuous function supported on $\overline{B}(\delta/2\pi)$ and hence, by Fourier inversion,
\begin{equation*}
\mc{J}(\z) = \int_{\overline{B}(\delta/2\pi)} \widehat{\mc{J}}(\t)\, e^{2\pi i \t \cdot \z}\, \d\t.
\end{equation*}
This shows that $|\mc{J}(\z)| \leq C e^{\delta |\y|}$, for $\z = \x + i\y$, and thus
\begin{equation} \label{Sec5_est_nec_PW_dist}
|\mc{L}_{\nu, N_{\nu}}(\delta,\lambda, \z)| \leq C (|\delta\z/2|^2 - a^2) e^{\delta |\y|}.
\end{equation}
The bound \eqref{Sec5_est_nec_PW_dist} allows us to invoke the Paley-Wiener theorem for distributions \cite[Theorem 1.7.7]{Hor} to conclude that the distributional Fourier transform of $\x \mapsto \mc{L}_{\nu, N_{\nu}}(\delta,\lambda, \x)$ is supported on $\overline{B}(\delta/2\pi)$ and thus \eqref{Sec5_dist_eq_for_L} holds. 
\end{proof}

\smallskip

\noindent {\it Step 2}. The case $N=1$.

\smallskip

By \eqref{Sec5_eq_def_L_nu_rad}, the entire function $\z \mapsto \L_{\nu, N_{\nu}}(\delta,\mu, \z)$ is radial when restricted to $\R^{N_{\nu}}$. This implies that 
$$\L_{\nu, N_{\nu}}(\delta,\mu, \z) = \widetilde{\L}_{\nu, N_{\nu}}(\delta,\mu, \z)$$
and Lemma \ref{type-to-ntype} (ii) implies that $\z \mapsto \L_{\nu, N_{\nu}}(\delta,\mu, \z)$ admits a power series representation as in \eqref{Sec4_PWRep}. Define the entire function $z \mapsto L_\nu(\delta,\mu,z)$ by 
\begin{align}\label{Sec5_L_nu_N1}
L_\nu(\delta,\mu,z) = \L_{\nu, N_{\nu}}(\delta,\mu,(z,0,\ldots,0)).
\end{align}
 By Lemma \ref{type-to-ntype} the function $z \mapsto L_\nu(\delta,\mu,z)$ has exponential type at most $\delta$, and we plainly observe that the analogues of \eqref{Sec5_cond1_import} and \eqref{Sec5_cond2_import} for $N=1$ hold for all real $x$. A change to radial variables in 
\eqref{Sec5_eva_case_N_nu} gives 
\begin{equation}\label{Sec5_eva_case_N_nu_d_1}
\int_{\R} \big\{g_\mu(x) - L_\nu(\delta,\mu,x)\big\} \,|x|^{2\nu+1} \,\dx = \int_0^\infty U_\nu^{1-}(\delta,\lambda)\, \d\mu(\lambda).
\end{equation}

\smallskip

\noindent {\it Step 3}. The case of general $N$.

\smallskip

We now define for $\z\in \C^N$
\begin{equation*}
\L_{\nu, N}(\delta,\mu, \z) = \psi_N(L_\nu(\delta,\mu,\cdot))(\z)
\end{equation*}
with $L_\nu(\delta,\mu,\cdot)$ from \eqref{Sec5_L_nu_N1}. By Lemma \ref{type-to-ntype} this function has exponential type at most $\delta$, and we observe again that the analogues of \eqref{Sec5_cond1_import} and \eqref{Sec5_cond2_import} continue to hold. It follows from \eqref{Sec5_eva_case_N_nu_d_1} and Theorem \ref{Thm2} that 
\begin{equation}\label{Sec5_eva_case_N_nu_d_N}
\int_{\R^N} \big\{\G_{\mu,N}(\x) - \L_{\nu,N}(\delta,\mu,\x)\big\} \,|\x|^{2\nu+2 - N} \d\x = \int_0^\infty U_\nu^{N-}(\delta,\lambda)\, \d\mu(\lambda).
\end{equation}

\smallskip

\noindent{\it Optimality}. 

\smallskip

Let $\mc{L} \in \E_{\delta}^{N-} (\mc{G}_{\mu,N})$ be such that 
\begin{equation}\label{Sec5_Opt_part_int_1}
\int_{\R^N} \big\{\mc{G}_{\mu,N}(\x) - \mc{L}(\x)\big\}\,|\x|^{2\nu+2-N}\,\d\x < \infty.
\end{equation}
By changing $\L$ by $\widetilde{\L}$ we may assume that $\L$ is radial on $\R^N$ (note that this does not affect the value on the integral \eqref{Sec5_Opt_part_int_1}). In this case, $\x \mapsto \{\L_{\nu, N}(\delta,\mu, \x) - \L(\x)\} \in L^1(\R^N, |\x|^{2\nu+2-N}\,\d\x)$. Let $L:\C \to \C$ be defined by $L(z) = \L((z,0, \ldots, 0))$. We want to show that 
\begin{equation}\label{Sec5_optimality_ntp}
\tfrac12 \omega_{N-1} \int_{-\infty}^{\infty} \big\{ L_\nu(\delta,\mu,x) - L(x) \big\}\,|x|^{2\nu +1}\,\dx = \int_{\R^N} \big\{\L_{\nu, N}(\delta,\mu, \x) - \mc{L}(\x)\big\}\,|\x|^{2\nu+2-N}\,\d\x  \geq 0. 
\end{equation}
Let 
$$G(z) = L_\nu(\delta,\mu,2z/\delta) - L(2z/\delta).$$ 
Then $G$ has exponential type at most $2$, and by \eqref{Lem17_i} it belongs to $L^1(\R, |E_\nu(x)|^{-2}\, \dx)$. We are now able to use Lemma \ref{Sec4_Cor22} to write $G = WW^* - TT^*$, with $W, T \in \H(E_{\nu})$. It follows from \eqref{Lem17_ii} and the fact that $\{z \mapsto K_{\nu}(\xi, z);\ A_{\nu}(\xi) = 0\}$ is a complete orthogonal set in $\H(E_{\nu})$ that
\begin{align*}
\int_{-\infty}^{\infty} G(x) \,|x|^{2\nu +1}\,\dx & = \int_{-\infty}^{\infty}  \big\{ |W(x)|^2 - |T(x)|^2\big\} \,|x|^{2\nu +1}\,\dx\\
& = c_{\nu}^{-1} \int_{-\infty}^{\infty}  \big\{ |W(x)|^2 - |T(x)|^2\big\} \,|E_{\nu}(x)|^{-2}\,\dx\\
& = c_{\nu}^{-1} \sum_{A_{\nu}(\xi) =0} \frac{|W(\xi)|^2 - |T(\xi)|^2}{K_{\nu}(\xi,\xi)} \\
& = c_{\nu}^{-1} \sum_{A_{\nu}(\xi) =0} \frac{G(\xi)}{K_{\nu}(\xi,\xi)}\\
& \geq 0.
\end{align*}
In the last step we have used the interpolating property \eqref{Sec5_cond2_import}, which translates to
\begin{equation*}
L_\nu(\delta,\mu,2\xi/\delta) = g_{\mu}(2\xi/\delta) \geq L(2\xi/\delta),
\end{equation*}
whenever $A_{\nu}(\xi) = 0$. This establishes \eqref{Sec5_optimality_ntp}.

\subsubsection{Proof of part {\rm (i)}} For $\kappa >0$, recall the definition of the measure $\mu_{\kappa}$ given in \eqref{def_mu_kappa}.  We observe from part (ii) of Theorem \ref{Thm3} and from part (i) of Theorem \ref{Thm2} that
\begin{align*}
U_\nu^{N-}(\delta, \mu) & =\int_0^\infty U_\nu^{N-}(\delta,\lambda)\, \d\mu(\lambda) = \int_0^\infty \kappa^{2\nu+2} \,U_\nu^{N-}(\kappa \delta, \kappa^2\lambda) \,\d\mu(\lambda) = \kappa^{2\nu+2} \int_0^\infty U_\nu^{N-} (\kappa\delta,\lambda)\,\d\mu_{\kappa^{-2}}(\lambda) \\
&=  \kappa^{2\nu+2}  \, U_\nu^{N-}(\kappa\delta,\mu_{\kappa^{-2}}),
\end{align*}
and analogously for the majorant case.

\section{Applications} \label{App}

\subsection{Some examples} We start with a collection of interesting examples of radial functions included in our Gaussian subordination framework.

\subsubsection{Functions with exponential subordination} \label{Comparison_Exp_sub} In \cite{CL3}, the extremal problem \eqref{Intro_EP1} - \eqref{Intro_EP2} was solved for a class of functions of the form
\begin{equation}\label{Sec6_exp_sub_class}
G(\x) = \int_0^{\infty} \big\{e^{-\tau |\x|} - e^{-\tau}\big\}\, \d\sigma (\tau),
\end{equation}
where $\sigma$ is a nonnegative Borel measure on $(0,\infty)$ satisfying 
\begin{equation}\label{Sec6_cond_min_exp_sub}
\int_0^{\infty} \frac{\tau}{1 + \tau^{2\nu +3}}\,\d\sigma(\tau) < \infty 
\end{equation}
for the minorant problem, and 
\begin{equation*}
\int_0^{\infty} \frac{\tau}{1 + \tau}\,\d\sigma(\tau) < \infty 
\end{equation*}
for the majorant problem. This includes examples such as $\x \mapsto -\log |\x|$ and $\x \mapsto \Gamma(\alpha +1) \big\{|\x|^{-\alpha-1}-1\big\}$ for $-2 < \alpha < 2\nu +1$. Let us show that these functions are included in our Gaussian subordination framework.

\smallskip

Let $N = \lceil 2\nu +2 \rceil$. We start with the well-known identity
$$e^{-|\x|} = \frac{1}{\sqrt{\pi}} \int_0^\infty \frac{e^{-u}}{\sqrt{u}} e^{-\frac{|\x|^2}{4u}} \du\,,$$
and a change of variables gives us
$$e^{-\tau |\x|} =\frac{\tau}{2\pi} \int_0^\infty e^{-\pi \lambda |\x|^2}   \, e^{-\tau^2/(4\pi\lambda)} \frac{\d\lambda}{\lambda^{3/2}}.$$
Taking Fourier transforms on both sides leads to
$$C_N \frac{\tau}{(\tau^2 + 4 \pi^2 |\t|^2)^{\frac{N+1}{2}}} = \frac{\tau}{2\pi} \int_0^\infty \lambda^{-N/2}\, e^{-\pi |\t|^2/\lambda} \,e^{-\tau^2/(4\pi\lambda)} \frac{\d\lambda}{\lambda^{3/2}},$$
where $C_N = 2^N \, \pi^{(N-1)/2}\, \Gamma(\tfrac{N+1}{2})$ (see \cite[p. 6]{SW}). A standard argument as in \eqref{Sec5_Dist_Arg} gives us
\begin{align*}
\int_{\R^{N}} G(\x)\,\widehat{\varphi}(\x)\,\d\x &= \int_{\R^{N} }\left(C_N \int_0^{\infty} \frac{\tau}{(\tau^2 + 4 \pi^2 |\t|^2)^{\frac{N+1}{2}}} \,\d\sigma(\tau)\right) \,\varphi(\t)\,\d\t\\
& = \int_{\R^{N} }\left(\int_0^\infty   \lambda^{-N/2}\, e^{-\pi |\t|^2/\lambda} \left(\int_0^\infty e^{-\tau^2/(4\pi\lambda)} \frac{\tau}{2\pi}\, \d\sigma(\tau)\right) \, \frac{\d\lambda}{\lambda^{3/2}}\right) \,\varphi(\t)\,\d\t
\end{align*}
for any Schwartz function $\varphi:\R^N \to \C$ vanishing in an open neighborhood of the origin. This establishes condition (Q4) with
\begin{equation}\label{Sec6_mu_com_from_exp}
\d\mu(\lambda) =  \left(\int_0^\infty e^{-\tau^2/(4\pi\lambda)} \frac{\tau}{2\pi}\, \d\sigma(\tau)\right) \, \frac{\d\lambda}{\lambda^{3/2}}.
\end{equation}
For the minorant problem we must verify that, under \eqref{Sec6_cond_min_exp_sub}, the measure $\d\mu$ defined by \eqref{Sec6_mu_com_from_exp} verifies \eqref{Intro_mu_1}. The proof for the majorant case will be analogous. In fact,
$$\int_0^1 \lambda\, \d\mu(\lambda) \leq \int_0^1 \left(\int_0^\infty e^{-\tau^2/(4\pi)} \frac{\tau}{2\pi}\, \d\sigma(\tau)\right) \, \frac{\d\lambda}{\lambda^{1/2}} < \infty,$$
and
\begin{align*}
\int_1^{\infty} \frac{1}{\lambda^{\nu +1}}\, \d\mu(\lambda) &= \int_1^{\infty}  \left(\int_0^1 e^{-\tau^2/(4\pi \lambda)} \frac{\tau}{2\pi}\, \d\sigma(\tau)\right) \, \frac{\d\lambda}{\lambda^{\nu +5/2}} + \int_1^{\infty}  \left(\int_1^{\infty} e^{-\tau^2/(4\pi \lambda)} \frac{\d\lambda}{\lambda^{\nu +5/2}}\right) \frac{\tau}{2\pi}\, \d\sigma(\tau)\\
& \leq \int_1^{\infty}  \left(\int_0^1  \frac{\tau}{2\pi}\, \d\sigma(\tau)\right) \, \frac{\d\lambda}{\lambda^{\nu +5/2}}  + \int_1^{\infty}  \left(\int_{0}^{\infty} e^{-1/(4\pi y)} \frac{\d y }{y^{\nu +5/2}}\right) \frac{\tau}{(2\pi)\, \tau^{2\nu +3}}\, \d\sigma(\tau) \\
& <\infty.
\end{align*}

\smallskip

Note that the Gaussian $x \mapsto e^{-\pi x^2}$ {\it is not included} in the class \eqref{Sec6_exp_sub_class}. In fact, suppose that 
\begin{equation}\label{Sec6_Gau_cannot}
e^{-\pi x^2} = \int_0^{\infty} \big\{e^{-\tau x} - e^{-\tau}\big\}\, \d\sigma (\tau)
\end{equation}
for $x>0$, where the measure $\sigma$ satisfies \eqref{Sec6_cond_min_exp_sub}. Since 
\begin{equation}\label{Sec6_gau_cannot2}
|e^{-\tau z} - e^{-\tau}| = \left| \int_1^z -\tau\, e^{-\tau s}\,\ds \right| \leq \tau\, e^{- c\tau}\, |z-1|
\end{equation}
for $\re(z) >0$, where  $c = \min\{\re(z), 1\}$, we see that both sides of \eqref{Sec6_Gau_cannot} have analytic continuations to the half-plane $\re(z) >0$. However, because of \eqref{Sec6_gau_cannot2}, the right-hand side grows linearly when $z = 1 + iy$, $y \to \infty$, while the left-hand side is given by $e^{-\pi (1 + 2iy - y^2)}$.

\smallskip

The previous discussion  shows that the class given by the Gaussian subordination method {\it is strictly larger} than the class \eqref{Sec6_exp_sub_class} obtained in \cite{CL3}, showing that this becomes indeed the most powerful framework to deal with this sort of extremal problems.

\subsubsection{Positive definite functions} We might consider a {\it finite} measure $\mu$ and the function given by
\begin{equation}\label{Sec6_Scho}
g_{\mu}(x) = \int_0^{\infty} e^{-\pi \lambda x^2}\,\dmu.
\end{equation}
It is straightforward to check that this function fits our Gaussian subordination framework and thus we can solve the extremal problem for its radial extension to $\R^N$, for any $N\geq 1$. Recall that, by a classical result of Schoenberg \cite[Theorems 2 and 3]{Scho}, a function $g$ admits the representation \eqref{Sec6_Scho} if and only if its radial extension to $\R^N$ is {\it positive definite} for all $N \geq 1$ or, equivalently, if the function $x \mapsto g(\sqrt{x})$, for $x >0$, is {\it completely monotone} (a function $f: [0,\infty) \to \R$ is said to be {\it completely monotone} if $(-1)^n\,f^{(n)}(t)\geq 0$ for $t>0$ and $n=0,1,2, \ldots,$ and $f(0) = f(0^+)$). Among such functions we highlight: 
\begin{itemize}
\item[(i)] $\x \mapsto e^{-\alpha |\x|^r}$, for $0 \leq r \leq 2$; 
\item[(ii)] $\x \mapsto (|\x|^2 + \alpha^2)^{-\beta}$, for $\alpha, \beta >0$; 
\item[(iii)] $\x \mapsto -\log[(|\x|^2 + \alpha^2)/(|\x|^2 + \beta^2)]$ for $0 < \alpha < \beta$.
\end{itemize}
In fact, the Gaussian subordination method also works if we take $\alpha =0$ in the family (iii) above. The one-dimensional extremal problem for the family (iii) was particularly important in bounding the modulus of the Riemann zeta-function on the critical strip under the Riemann hypothesis \cite{CC, CS}. 

\subsubsection{Power functions} \label{power_functions} Let $\gamma(s) = \pi^{-s/2}\,\Gamma(\tfrac{s}{2})$. The function $\x \mapsto \gamma(-\sigma)\,|\x|^{\sigma}$ for $\sigma > - 2\nu -2$ (minorant) and $\sigma > 0$ (majorant), with $\sigma \neq 0,2, 4, \ldots$, also falls under the scope of our Gaussian subordination method. In fact, in dimension $N_{\nu} = \lceil 2\nu + 2 \rceil$, the distributional Fourier transform of this function outside the origin is given by $\gamma(N_{\nu} + \sigma)\,|\t|^{-N_{\nu} - \sigma}$ (see \cite[Chapter III, Theorem 5, Eq. (33)]{St}) and we have
\begin{equation}\label{power_kernel_FT_Gaussian}
\gamma(N_{\nu} + \sigma)\,|\t|^{-N_{\nu} - \sigma} = \int_0^\infty \lambda^{-N_\nu/2} \, e^{-\pi|\t|^2/\lambda} \,\d\mu_{\sigma}(\lambda)
\end{equation}
for $\d\mu_{\sigma}(\lambda) = \lambda^{-\frac{\sigma}{2} - 1} \,\d\lambda$. It is clear that this measure satisfies \eqref{Intro_mu_1} (minorant) and \eqref{Intro_mu_2} (majorant).

\subsection{Hilbert-type inequalities} The solution of our extremal problem is connected with inequalities of Hilbert-type. This was previously observed, for instance, in \cite{CL3, CLV, GV, V}. For an account on the original Hilbert's inequality and its weighted version, see the work of Montgomery and Vaughan \cite{MV}. In our setting, we let $N \geq 1$ be a given dimension, and consider the case $2\nu + 2 = N$. For a nonnegative Borel measure $\mu$ on $(0,\infty)$ we define
\begin{equation}\label{Hilbert_ineq_def_Q}
\mc{Q}_{\mu}(\y) = \int_0^\infty \lambda^{-N/2} \, e^{-\pi|\y|^2/\lambda} \,\d\mu(\lambda).
\end{equation}
Note that the integrand is the Fourier transform of the Gaussian $\mc{F}_{\lambda}(\x) = e^{-\pi \lambda |\x|^2}$ in dimension $N$.

\begin{theorem}\label{Thm_Hilbert}
Let $\delta >0$, $N \geq 1$ and $2\nu + 2 = N$. Let $\{a_j\}_{j=1}^M$ be a sequence of complex numbers and $\{\y_j\}_{j=1}^M$ be a sequence of well-spaced vectors in $\R^N$, in the sense that $|\y_j - \y_l|\geq \delta$ for any $j \neq l$. The following propositions hold:
\smallskip
\begin{enumerate}
\item[(i)] If $\mu$ is a nonnegative Borel measure on $(0,\infty)$ satisfying \eqref{Intro_mu_1} and $\mc{Q}_{\mu}$ is defined as in \eqref{Hilbert_ineq_def_Q} then
\begin{equation}\label{Hilb_eq0}
-U_{\nu}^{N-}(2\pi\delta, \mu)\,\sum_{j=1}^M |a_j|^2\leq \sum_{\stackrel{j,l=1}{j\neq l}}^{M} a_j\,\overline{a_l} \,\mc{Q}_{\mu}(\y_j - \y_l).
\end{equation}
\item[(ii)] If $\mu$ is a nonnegative Borel measure on $(0,\infty)$ satisfying \eqref{Intro_mu_2} and $\mc{Q}_{\mu}$ is defined as in \eqref{Hilbert_ineq_def_Q} then
\begin{equation}\label{Hilb_eq1}
\sum_{\stackrel{j,l=1}{j\neq l}}^{M} a_j\,\overline{a_l} \,\mc{Q}_{\mu}(\y_j - \y_l) \leq U_{\nu}^{N+}(2\pi\delta,\mu)\,\sum_{j=1}^M |a_j|^2.
\end{equation}
\end{enumerate}
\end{theorem}

\begin{proof}
The minorant $\x \mapsto \mc{L}_{\nu, N}(2\pi \delta, \lambda, \x)$ produced in Theorem \ref{Thm2} is integrable and of exponential type $2\pi\delta$. Thus its Fourier transform is supported on $\overline{B}(\delta)$. We now observe that 
\begin{align*}
0 & \leq \int_{-\infty}^{\infty} \big\{\mc{F}_{\lambda}(\x) - \mc{L}_{\nu, N}(2\pi \delta, \lambda, \x)\big\}\left| \sum_{j=1}^m a_j \,e^{-2\pi i \y_j \cdot \x} \right|^2 \d\x\\
& = \sum_{j,l=1}^m a_j\,\overline{a_l}  \int_{-\infty}^{\infty} \big\{\mc{F}_{\lambda}(\x) - \mc{L}_{\nu, N}(2\pi \delta, \lambda, \x)\big\}   \,e^{-2\pi i (\y_j - \y_l) \cdot \x} \,  \d\x.
\end{align*}
This leads us to 
\begin{equation}\label{Hilb_eq2}
-U_{\nu}^{N-}(2\pi\delta, \lambda) \sum_{j=1}^M |a_j|^2 \leq \sum_{\stackrel{j,l=1}{j\neq l}}^{M} a_j\,\overline{a_l} \,\widehat{\mc{F}}_{\lambda}(\y_j - \y_l).
\end{equation}
Integrating both sides of \eqref{Hilb_eq2} against $\dmu$, we arrive at \eqref{Hilb_eq0}. The proof for the upper bound \eqref{Hilb_eq1} is analogous, now using the majorant produced in Theorem \ref{Thm2}. 
\end{proof}

An interesting example arises when we consider the measures $\d\mu_{\sigma}(\lambda) = \lambda^{-\frac{\sigma}{2} - 1} \,\d\lambda$ from \S \ref{power_functions}. The identity \eqref{power_kernel_FT_Gaussian} remains valid as long as $\sigma > - N$ and Theorem \ref{Thm_Hilbert} gives us
\begin{equation}\label{disc_HPL}
-\frac{U_{\nu}^{N-}(2\pi\delta,\mu_{\sigma})}{\pi^{-(N+ \sigma)/2}\,\Gamma\left(\frac{N + \sigma}{2}\right)}\,\sum_{j=1}^M |a_j|^2\ \leq\ \sum_{\stackrel{j,l=1}{j\neq l}}^{M} \frac{ a_j\,\overline{a_l}}{|\y_j - \y_l|^{N + \sigma}} \ \leq\  \frac{U_{\nu}^{N+}(2\pi\delta,\mu_{\sigma})}{\pi^{-(N+ \sigma)/2}\,\Gamma\left(\frac{N + \sigma}{2}\right)}\,\sum_{j=1}^M |a_j|^2.
\end{equation}
The lower bound holds as long as $\sigma > -N$, while the upper bound holds as long as $\sigma >0$. This inequality is related to the discrete Hardy-Littlewood-Sobolev inequality \cite[p. 288]{HPL}, and improves the result of \cite{CL3}, where the restriction $-N < \sigma < 1$ had to be imposed. The one-dimensional version of \eqref{disc_HPL} appeared in \cite[Corollary 22]{CLV}.

\subsection{Periodic analogues} The interpolation tools developed in Section \ref{Int_type_zero} allow us to treat the problem of optimal one-sided approximation of even periodic functions by trigonometric polynomials. This involves the same circle of ideas discussed in \cite[Section 8]{CL3}, in which the exponential replaces the Gaussian as the base function, and we refer to this work whenever possible. As in the Euclidean setting, the Gaussian subordination method also produces {\it a strictly larger class} in the periodic setting, when compared to the exponential method of \cite{CL3}.

\smallskip

We denote by $\mathbb{D} \subset \C$ the open unit disc and by $\partial \mathbb{D}$ the unit circle. As in \cite{CL3}, we identify measures on $\R/\Z$ with measures on the unit circle $\uc$ via the map $x \mapsto e^{2\pi i x}$. Throughout this section we let $\vartheta$ be a nontrivial {\it even} probability measure on $\R/\Z$ (recall that a measure on $\R/\Z$ is trivial if it has finite support), and thus the corresponding nontrivial probability measure on the unit circle $\uc$ (that we keep calling $\vartheta$) satisfies
\[
\vartheta(A) = \vartheta(\overline{A})
\]
for any Borel set $A\subseteq \uc$, where $\overline{A} = \{\overline{z}:z\in A\}$. For $\im(\tau) >0$, we consider the Jacobi theta-function
\[
\theta_3(z,\tau) = \sum_{n=-\infty}^\infty e^{\pi i \tau n^2 + 2\pi i n z}.
\]
Poisson summation implies, for $\lambda>0$, that
\begin{align}\label{Per_Poisson_sum}
\lambda^{-\frac12} \theta_3\big(z,i\lambda^{-1} \big) = \sum_{j = -\infty}^\infty e^{-\pi \lambda(j-z)^2}.
\end{align}
This is the periodization of the Gaussian $\mc{F}_{\lambda}(x) = e^{-\pi \lambda x^2}$. We first address the problem of majorizing and minorizing  this function by trigonometric polynomials of a given degree $n$, in a way to minimize the $L^1(\R/\Z,\d\vartheta)$-error. This problem with respect to the Lebesgue measure was treated in \cite{CLV}. Here, of course, a trigonometric polynomial $m(z)$ of degree at most $n$ is a $1$-periodic function of the form 
$$m(z) = \sum_{k=-n}^{n} a_k  e^{2 \pi i k z},$$
where $a_k \in \C$. We say the $m(z)$ is a real trigonometric polynomial if it is real for real $z$. Let us denote the space of trigonometric polynomials of degree at most $n$ by $\Lambda_n$. 

\subsubsection{One-sided approximations to theta-functions} We first state the analogue of \cite[Lemma 28]{CL3}. 

\begin{lemma} Let $G$ be an even, $1$-periodic, non-constant Laguerre-P\'olya function of exponential type $\tau(G)$. Let $z \mapsto L(G, \lambda, z)$ and $z \mapsto M(G, \lambda, z)$ be defined as in Lemma \ref{L-thm}.

\begin{itemize}
\item[(i)] If $G(0)>0$, the function 
\[
\ell(G,\lambda,z) = \sum_{j\in\Z} L(G,\lambda,z+j)
\]
is a real trigonometric polynomial in $z$ of degree less than $\tau(G)/(2\pi)$ satisfying
\begin{align}
G(x) \Big\{\lambda^{-\frac12} \theta_3\big(x, i\lambda^{-1}\big) - \ell(G,\lambda,x)  \Big\}\ge 0
\end{align}
for all $x\in\R$, and 
\[
\ell(G,\lambda,\xi) = \lambda^{-\frac12} \theta_3\big(\xi,i\lambda^{-1}\big)
\]
for all $\xi \in \R$ with $G(\xi) =0$. 

\smallskip

\item[(ii)] If $G$ has a double zero at the origin and $G(0^+)>0$, the function 
\[
m(G,\lambda,z) = \sum_{j\in\Z} M(G,\lambda,z+j)
\]
is a real trigonometric polynomial in $z$ of degree less than $\tau(G)/(2\pi)$ satisfying
\begin{align}
G(x) \Big\{ m(G,\lambda,x) - \lambda^{-\frac12} \theta_3\big(x, i\lambda^{-1}\big)\Big\} \ge 0
\end{align}
for all $x\in\R$, and 
\[
m(G,\lambda,\xi) = \lambda^{-\frac12} \theta_3\big(\xi,i\lambda^{-1}\big)
\]
for all $\xi \in \R$ with $G(\xi) =0$. 
\end{itemize} \end{lemma}

\begin{proof} The proof of \cite[Lemma 28]{CL3} uses that $L$ has exponential type and is in $L^1(\R)$, that it minorizes $x\mapsto e^{-\lambda|x|}$ on the real line, and that it interpolates at the zeros of $G$. When replacing the exponential with the Gaussian, these assumptions are given for $L$ in our setting, hence the proof goes through without any changes; likewise for the majorant. 
\end{proof}

In what follows, if an algebraic polynomial $P(z)$ has degree $n$, we write $P^*(z) = z^n \, \overline{P\big(\bar{z}\,^{-1}\big)}$. Let $\{\varphi_n\}_{n=0}^{\infty}$ be the sequence of {\it orthonormal polynomials} in the unit circle with respect to the inner product
$$\langle f,g\rangle_{L^2(\uc, \d\vartheta)} = \int_{\uc} f(z)\,\overline{g(z)}\, \d\vartheta(z).$$
We note that, since the measure $\vartheta$ is even, the algebraic polynomials $\varphi_n$ have real coefficients (see \cite[Lemma 25]{CL3}). Define the algebraic polynomials
\begin{align}
A_{n+1}(z) &:= \frac{1}{2}\big\{\varphi_{n+1}^*(z) + \varphi_{n+1}(z)\big\},\label{Sec8_DefAn}\\
B_{n+1}(z) &:= \frac{i}{2}\big\{ \varphi_{n+1}^*(z) - \varphi_{n+1}(z)\big\},\label{Sec8_DefBn}
\end{align}
which have degree $n+1$ and have only simple zeros lying on the unit circle. Also, define the function
\begin{equation}\label{Sec8_defKn}
K_{n}(w,z) =\frac{\varphi^*_{n+1}(z)\, \overline{\varphi_{n+1}^*(w)} - \varphi_{n+1}(z)\,\overline{\varphi_{n+1}(w)}}{1-\bar{w}z}\,,
\end{equation}
which is the reproducing kernel of the space $\mathcal{P}_n$ of algebraic polynomials with degree at most $n$, with inner product $\langle \cdot,\cdot\rangle_{L^2(\uc,\d\vartheta)}$ (for details see \cite[Sections 8.1 and 8.2]{CL3}). With these definitions we can state our next result.

\begin{theorem}\label{Sec8_Thm27}
Let $n \in \Z^+$ and $\vartheta$ be a nontrivial {\it even} probability measure on $\R/\Z$. Let $\varphi_{n+1}(z) = \varphi_{n+1}(z;\d\vartheta)$ be the $(n+1)$-th orthonormal polynomial in the unit circle with respect to this measure and consider $K_n, A_{n+1}, B_{n+1}$ as defined in \eqref{Sec8_DefAn}, \eqref{Sec8_DefBn} and  \eqref{Sec8_defKn}.  Let $\mc{A}_{n+1} =\big\{\xi\in \R/\Z: A_{n+1}(e^{2\pi i \xi}) =0\big\}$ and $\mc{B}_{n+1} = \big\{\xi\in\R/\Z: B_{n+1}(e^{2\pi i \xi}) =0\big\}$.

\smallskip

\begin{itemize}
\item[(i)] If $\ell \in \Lambda_n$ satisfies 
\begin{equation}\label{Sec8_ineq_min1}
\ell(x) \leq \lambda^{-\frac12} \theta_3\big(x,i\lambda^{-1} \big) 
\end{equation} 
for all $x \in \R$ then
\begin{equation}\label{Sec8_ineq_min2}
\int_{\R/\Z} \ell(x) \, \d\vartheta(x) \leq \sum_{\xi \in \mc{A}_{n+1}} \frac{\lambda^{-\frac12} \theta_3\big(\xi,i\lambda^{-1} \big) }{K_{n}(e^{2\pi i \xi},e^{2\pi i \xi})}.
\end{equation}
Moreover, there exists a unique real trigonometric polynomial $z \mapsto \ell_{\vartheta}(n, \lambda, z) \in \Lambda_n$ satisfying \eqref{Sec8_ineq_min1} for which the equality in \eqref{Sec8_ineq_min2} holds.

\smallskip

\item[(ii)] If $m \in \Lambda_n$ satisfies 
\begin{equation}\label{Sec8_ineq_maj1}
m(x) \geq \lambda^{-\frac12} \theta_3\big(x,i\lambda^{-1} \big) 
\end{equation} 
for all $x \in \R$ then
\begin{equation}\label{Sec8_ineq_maj2}
\int_{\R/\Z} m(x) \, \d\vartheta(x) \geq \sum_{\xi \in \mc{B}_{n+1}} \frac{\lambda^{-\frac12} \theta_3\big(\xi,i\lambda^{-1} \big) }{K_{n}(e^{2\pi i \xi},e^{2\pi i \xi})}.
\end{equation}
Moreover, there exists a unique real trigonometric polynomial $z \mapsto m_{\vartheta}(n, \lambda, z) \in \Lambda_n$ satisfying \eqref{Sec8_ineq_maj1} for which the equality in \eqref{Sec8_ineq_maj2} holds.
\end{itemize}
\end{theorem} 

\begin{proof}
As is the case with the previous lemma, the proof of Theorem \ref{Sec8_Thm27} carries over verbatim from \cite[Theorem 27]{CL3}, and we omit the details.
\end{proof}

\subsubsection{One-sided approximations to general even periodic functions} Our goal now is to integrate the parameter $\lambda$ against suitable measures and solve the extemal problem for a class of even periodic functions. We consider two classes of nonnegative Borel measures $\varsigma$  on $(0,\infty)$. For the minorant problem we require that
\begin{equation}\label{Sec8_meas1}
\int_0^\infty  \lambda^{-\frac12} \,e^{-\pi \lambda^{-1} -  a \lambda}\, \d\varsigma(\lambda) < \infty
\end{equation}
{\it for any} $a >0$, whereas for the majorant problem we require the more restrictive condition
\begin{equation}\label{Sec8_meas2}
\int_0^{\infty}\big(1 + \lambda^{-\frac12} \big)\,e^{-\pi \lambda^{-1}}  \,\d\varsigma(\lambda) < \infty.
\end{equation}
In this subsection we majorize and minorize the periodic function
\begin{equation}\label{def_h_varsigma}
h_{\varsigma}(x) := \int_0^{\infty} \Big\{\lambda^{-\frac12} \theta_3\big(x,i\lambda^{-1} \big)-\lambda^{-\frac12} \theta_3\big(\hh,i\lambda^{-1} \big)\Big\}\,\d\varsigma(\lambda)
\end{equation}
by trigonometric polynomials of a given degree $n$, minimizing the $L^1(\R/\Z, \d\vartheta)$-error. Note the convenient subtraction of the term $\lambda^{-\frac12} \theta_3(\hh,i\lambda^{-1} )$ to generate a better decay rate in $\lambda$ as $\lambda \to 0$. If $\varsigma$ satisfies \eqref{Sec8_meas1}, $h_{\varsigma}$ is well-defined and finite for all $x \notin \Z$ (we may have $h_{\varsigma}(0) = \infty$), and if $\varsigma$ satisfies \eqref{Sec8_meas2}, $h_{\varsigma}$ is well-defined and finite for all $x \in \R$. We obtain the following result.

\begin{theorem}\label{Sec8_Thm29}
Let $n \in \Z^{+}$ and $\vartheta$ be a nontrivial {\it even} probability measure on $\R/\Z$. Let $\varphi_{n+1}(z) = \varphi_{n+1}(z;\d\vartheta)$ be the $(n+1)$-th orthonormal polynomial in the unit circle with respect to this measure and consider $K_n, A_{n+1}, B_{n+1}$ as defined in \eqref{Sec8_DefAn}, \eqref{Sec8_DefBn} and  \eqref{Sec8_defKn}. Let $\mc{A}_{n+1} = \big\{\xi\in \R/\Z: A_{n+1}(e^{2\pi i \xi}) =0\big\}$ and $\mc{B}_{n+1} = \big\{\xi\in\R/\Z: B_{n+1}(e^{2\pi i \xi}) =0\big\}$.
\smallskip
\begin{itemize}
\item[(i)] Let $\varsigma$ satisfy \eqref{Sec8_meas1}. If $\ell \in \Lambda_n$ satisfies 
\begin{equation}\label{Sec8_ineq_min1-thm29}
\ell(x) \leq h_\varsigma(x)
\end{equation} 
for all $x \in \R$ then
\begin{equation}\label{Sec8_ineq_min2-thm29}
\int_{\R/\Z} \ell(x) \, \d\vartheta(x) \leq \sum_{\xi \in \mc{A}_{n+1}} \frac{h_{\varsigma}(\xi)}{K_{n}(e^{2\pi i \xi},e^{2\pi i \xi})}.
\end{equation}
Moreover, there exists a unique real trigonometric polynomial $z \mapsto \ell_{\vartheta}(n, \varsigma, z) \in \Lambda_n$ satisfying \eqref{Sec8_ineq_min1-thm29} for which the equality in \eqref{Sec8_ineq_min2-thm29} holds.

\smallskip

\item[(ii)] Let $\varsigma$ satisfy \eqref{Sec8_meas2}. If $m \in \Lambda_n$ satisfies 
\begin{equation}\label{Sec8_ineq_maj1-thm29}
m(x) \geq h_{\varsigma}(x)
\end{equation} 
for all $x \in \R$ then
\begin{equation}\label{Sec8_ineq_maj2-thm29}
\int_{\R/\Z} m(x) \, \d\vartheta(x) \geq \sum_{\xi \in \mc{B}_{n+1}} \frac{h_{\varsigma}(\xi)}{K_{n}(e^{2\pi i \xi},e^{2\pi i \xi})}.
\end{equation}
Moreover, there exists a unique real trigonometric polynomial $z \mapsto m_{\vartheta}(n, \varsigma, z) \in \Lambda_n$ satisfying \eqref{Sec8_ineq_maj1-thm29} for which the equality in \eqref{Sec8_ineq_maj2-thm29} holds.
\end{itemize}
\end{theorem}

\begin{proof} {\it Part} (i). Write 
$$h_{\lambda}(x) = \lambda^{-\frac12} \theta_3\big(x,i\lambda^{-1} \big)-\lambda^{-\frac12} \theta_3\big(\hh,i\lambda^{-1} \big).$$
From Theorem \ref{Sec8_Thm27}, the optimal trigonometric polynomial of degree at most $n$ that minorizes $h_{\lambda}$ is 
$$\widetilde{\ell}_{\vartheta}(n,\lambda,z) := \ell_{\vartheta}(n,\lambda,z) - \lambda^{-\frac12} \theta_3\big(\hh,i\lambda^{-1} \big).$$
Let us write
$$\widetilde{\ell}_{\vartheta}(n,\lambda,z) = \sum_{k=-n}^{n} a_k(n,\lambda)\, e^{2\pi i k z},$$
where $a_k = a_{k,\vartheta}$. From the interpolation properties we have, for all $\lambda >0$, 
\begin{align}
\widetilde{\ell}_{\vartheta}(n,\lambda,\xi) &=  \sum_{k=-n}^{n} a_k(n,\lambda)\, e^{2\pi i k \xi}=h_\lambda(\xi),\label{Sec_per_eq_1_ak}\\
\widetilde{\ell}_{\vartheta}\,'(n,\lambda,\xi) & = \sum_{k=-n}^{n} 2\pi i k \,a_k(n,\lambda)\, e^{2\pi i k \xi} = h'_{\lambda}(\xi) \label{Sec_per_eq_2_ak}
\end{align}
for all $\xi \in \mc{A}_{n+1}$ (note that $\xi \neq 0$). Since we have $2n+1$ coefficients $\{a_k; \, -n \leq k \leq n\}$ and $2n+2$ equations, this is an overdetermined system. We also know that this interpolation problem has a unique solution, so we can drop the last equation and invert the coefficient matrix to obtain each $a_k$ as a function of the values $\{h_{\lambda}(\xi), \, h'_{\lambda}(\xi);\, \xi \in \mc{A}_{n+1}\}$. From the definition of $h_{\lambda}$ and identity \eqref{Per_Poisson_sum}, together with \eqref{Sec8_meas1}, we find that 
\begin{align}\label{Sec_per_eq_3_ak}
\int_0^{\infty} |h_{\lambda}(\xi)| \, \d\varsigma(\lambda) < \infty
\end{align}
and
\begin{align}\label{Sec_per_eq_4_ak}
\int_0^{\infty} |h'_{\lambda}(\xi)| \, \d\varsigma(\lambda) < \infty
\end{align}
for $\xi \neq 0$. We then conclude that each coefficient $a_k(n,\lambda)$ is absolutely integrable with respect to $\d\varsigma(\lambda)$ and we are able to define
\begin{equation*}
\ell_{\vartheta}(n,\varsigma, z) := \int_0^{\infty} \widetilde{\ell}_{\vartheta}(n,\lambda, z)\,\d\varsigma(\lambda).
\end{equation*}
It is clear that 
$$\ell_{\vartheta}(n,\varsigma, x) \leq h_{\varsigma}(x)$$
for all $x \in \R$ and that 
$$\ell_{\vartheta}(n,\varsigma, \xi) =  h_{\varsigma}(\xi)$$
for all $\xi \in \mc{A}_{n+1}$. The optimality and uniqueness of this minorant follows as in \cite[Section 8]{CL3}.

\smallskip

\noindent{\it Part} (ii). Observe now that $0 \in \mc{B}_{n+1}$ and, when solving the system of equations for $a_k(n,\lambda)$, we may throw away the equation corresponding to $\xi = 0$ in \eqref{Sec_per_eq_2_ak} (which might not hold if $h_{\lambda}$ is not differentiable at the origin). Relation \eqref{Sec_per_eq_3_ak} continues to hold for all $\xi \in \R/\Z$ (due to the more restrictive condition \eqref{Sec8_meas2}) while \eqref{Sec_per_eq_4_ak} holds for $\xi \neq 0$. The rest is analogous.
\end{proof}

\noindent {\sc Remark:} In \cite[Theorem 29]{CL3} the analogous version of Theorem \ref{Sec8_Thm29} was proved for even periodic functions with an exponential subordination. Theorem \ref{Sec8_Thm29} is strictly more powerful than \cite[Theorem 29]{CL3} in the sense that every function that fits the exponential subordination framework of \cite[Theorem 29]{CL3} also fits the Gaussian subordination framework of Theorem \ref{Sec8_Thm29}, and there are functions (e.g. the theta-functions \eqref{Per_Poisson_sum}) that fit the Gaussian subordination framework but do not fit the exponential subordination framework of \cite[Theorem 29]{CL3}. Verifications of these claims follow the same circle of ideas described in \S \ref{Comparison_Exp_sub} and are left to the interested reader.

\section*{Acknowledgements}
 E.C. acknowledges support from CNPq-Brazil grants $302809/2011-2$ and $477218/2013-0$, and FAPERJ grant $E-26/103.010/2012$. We are thankful to IMPA in Rio de Janeiro and to the Hausdorff Institute for Mathematics in Bonn for supporting research visits during the preparation of this work. We are also thankful to the anonymous referee for providing valuable suggestions to clarify the exposition in some passages of the paper.

\end{document}